\documentclass[12pt, reqno]{amsart}
\usepackage{amsmath, amstext, amsbsy, amssymb}

\newtheorem{theorem}{Theorem}[section]
\newtheorem{lemma}[theorem]{Lemma}
\newtheorem{proposition}[theorem]{Proposition}
\newtheorem{corollary}[theorem]{Corollary}
\theoremstyle{definition}
\newtheorem{definition}[theorem]{Definition}

\theoremstyle{remark}
\newtheorem{remark}[theorem]{Remark}
\theoremstyle{conjecture}

\theoremstyle{problem}

\numberwithin{equation}{section}

\newcommand{\ch}{\mbox{ch} }
\newcommand{\C}{ \Bbb C }

\newcommand{\ep}{\epsilon}

\newcommand{\Pee}{ \mathbb P }

\newcommand{\W}{\widetilde}
\newcommand{\w}{\tilde}

\newcommand{\Z}{ \mathbb Z }

\newcommand\Hom{\operatorname{Hom}}

\newcommand\Pic{\operatorname{Pic}}

\newcommand\Supp{\operatorname{Supp}}

\newcommand\Ext{\operatorname{Ext}}

\newcommand\Hilb{\operatorname{Hilb}}
\newcommand\rank{\operatorname{rank}}
\newcommand\Quot{\operatorname{Quot}}

{\vskip-\lastskip\medskip
  \noindent
  {\em #1.}\enspace
  }%
{\qed\par\medskip
  }

\begin{document}

\title[Donaldson-Thomas invariants of certain Calabi-Yau $3$-folds]
{Donaldson-Thomas invariants of certain Calabi-Yau $3$-folds}

\author[Wei-Ping Li]{Wei-Ping Li$^1$}
\address{Department of Mathematics, HKUST, Clear Water Bay,
Kowloon, Hong Kong} \email{mawpli@ust.hk}
\thanks{${}^1$Partially supported by the grants   GRF601808}

\author[Zhenbo Qin]{Zhenbo Qin$^2$}
\address{Department of Mathematics, University of Missouri,
Columbia, MO 65211, USA} \email{qinz@missouri.edu}
\thanks{${}^2$Partially supported by an NSF grant}


\begin{abstract}
We compute the Donaldson-Thomas invariants for two types of Calabi-Yau 
$3$-folds. These invariants are associated to the moduli spaces of 
rank-$2$ Gieseker semistable sheaves. None of the sheaves are locally free, 
and their double duals are locally free stable sheaves investigated earlier
in \cite{DT, Tho, LQ2}. We show that these Gieseker moduli spaces are 
isomorphic to some Quot-schemes. We prove a formula for Behrend's 
$\nu$-functions when torus actions present with positive dimensional fixed point sets, and use it to obtain 
the generating series of the relevant Donaldson-Thomas invariants 
in terms of the McMahon function. Our results might shed some light on 
the wall-crossing phenomena of Donaldson-Thomas invariants.
\end{abstract}

\maketitle

\section{\bf Introduction}
\label{sect_intro}
The Donaldson-Thomas invariants of a Calabi-Yau $3$-fold $Y$ essentially 
count the number of stable sheaves on $Y$. 
It attracts intensive activities recently due to the conjectural 
relations with Gromov-Witten invariants proposed by
Maulik, Nekrasov, Okounkov and Pandharipande. The moduli space of 
stable sheaves in \cite{MNOP} consists of ideal sheaves
defining $1$-dimensional closed subschemes of $Y$ with 
some $0$-dimensional components and some embedded points.
The Donaldson-Thomas invariants associated to the $0$-dimensional closed 
subschemes of $Y$ have been determined in \cite{JLi, BF, LP}.
Curves on $Y$ may also be related to rank-$2$ vector bundles via 
the Serre construction. In fact, the first example of Donaldson-Thomas 
invariants in \cite{DT, Tho} counts certain stable rank-$2$ sheaves
on the Calabi-Yau $3$-fold $Y$ which is the smooth intersection of 
a quartic hypersurface and a quadric hypersurface in $\mathbb P^5$.
In \cite{LQ2}, we studied some moduli spaces of rank-$2$ stable sheaves
on a Calabi-Yau hypersurface of $\mathbb P^1\times\mathbb P^1\times \mathbb P^n$ 
for $n\ge 2$ and computed the corresponding Donaldson-Thomas invariants 
of the $3$-fold $Y$ when $n=2$. The idea in \cite{DT, Tho} and \cite{LQ2}
is to give a complete description of the moduli spaces. 
Let $L$ (respectively, ${\bf c}_0$) be the ample line bundle 
(respectively, total Chern class) considered in \cite{DT, Tho} or \cite{LQ2}. 
Then the moduli space $\overline{\mathfrak M}_L({\bf c}_0)$ 
of rank-$2$ Gieseker $L$-semistable sheaves with total Chern classes 
${\bf c}_0$ consists of two smooth points for the case in \cite{DT, Tho} 
and is a projective space for the case in \cite{LQ2}. 
It turns out that in both cases, all the sheaves 
in $\overline{\mathfrak M}_L({\bf c}_0)$ are stable and locally free. Here ${\bf c}_0$ is chosen 
as in Theorem A and Theorem B. For a general total Chern class ${\bf c}_0$, a
full description of the moduli space $\overline{\mathfrak M}_L({\bf c}_0)$ is yet to be done. 

In this paper, we compute the Donaldson-Thomas invariant, denoted by 
$\lambda(L, {\bf c}_m)$, associated to the Gieseker moduli space 
$\overline{\mathfrak M}_L({\bf c}_m)$ where 
$$
{\bf c}_m = {\bf c}_0 - m[y_0]
$$ 
and $y_0 \in Y$ is a fixed point. Again, an important ingredient is to 
understand $\overline{\mathfrak M}_L({\bf c}_m)$.
We show that whenever $\overline{\mathfrak M}_L({\bf c}_m)$ with $m \ne 0$
is not empty, none of the sheaves $E \in \overline{\mathfrak M}_L({\bf c}_m)$
are locally free but their double duals $E^{**}$ are locally free 
and contained in the moduli space $\overline{\mathfrak M}_L({\bf c}_0)$
considered in \cite{DT, Tho} and \cite{LQ2}. 

More precisely, for the pair $(L, {\bf c}_0)$ from \cite{DT, Tho},
let the two smooth points in $\overline{\mathfrak M}_L({\bf c}_0)$ 
be represented by the rank-$2$ stable bundles $E_{0, 1}$ and $E_{0, 2}$. 
If $\overline{\mathfrak M}_L({\bf c}_m) \ne \emptyset$, then $m$ is even
and nonnegative. When $m \ge 0$, the moduli space 
$\overline{\mathfrak M}_{L}({\bf c}_{2m})$ is isomorphic to the disjoint 
union of the Quot-schemes $\Quot^m_{E_{0, 1}}$ and $\Quot^m_{E_{0, 2}}$.

\medskip\noindent
{\bf Theorem A.}  
{\it Let $Q_0$ be a smooth quadric in $\mathbb P^5$, $H$ be a hyperplane in 
$\mathbb P^5$, $P$ be a plane on $Q_0$, and $Y$ be the Calabi-Yau $3$-fold 
which is the smooth intersection of a quartic hypersurface and $Q_0$ in 
$\mathbb P^5$. Let the total Chern class be
$${\bf c}_0=1+H|_Y+P|_Y.$$
Let $\chi(Y)$ be the Euler characteristic of $Y$, and $M(q) = 
\prod_{m=1}^{+\infty} {1 \over (1-q^m)^m}$ be the McMahon function. Then, 
\begin{eqnarray*}    
  \sum_{m \in \Z} \lambda(L, {\bf c}_m) \, q^m
= 2 \cdot M(q^2)^{2 \, \chi(Y)}.         
\end{eqnarray*}}

Next, let $Y$ be the Calabi-Yau $3$-fold considered in \cite{LQ2},
i.e., $Y$ is a smooth Calabi-Yau hypersurface of $\mathbb P^1\times
\mathbb P^1 \times \mathbb P^2$. Let $L_r^Y= L =  
\pi_2^*\mathcal O_{\mathbb P^1}(1) \otimes 
\pi_3^*\mathcal O_{\mathbb P^2}(r)|_Y$ be a $\mathbb Q$-line bundle on $Y$
where $\pi_i$ is the $i$-th projection of $\mathbb P^1\times
\mathbb P^1 \times \mathbb P^2$. Let 
$$
2(2-\ep_2)/(2+\ep_1) < r < 2(2-\ep_2)/\ep_1
$$
where $\epsilon_1, \epsilon_2 = 0, 1$ appear in the definition of 
the total Chern class ${\bf c}_{m}$ in (\ref{B.1}). In \cite{LQ2},
we proved that $\overline{\mathfrak M}_{L_r^Y}({\bf c}_0)$ is 
isomorphic to a projective space and consists of stable bundles.
Let $\mathcal E_0$ be a universal vector bundle over 
$\overline{\mathfrak M}_{L_r^Y}({\bf c}_{0})\times Y$, and let
$$
\Quot^m_{\mathcal E_0/} =
\Quot^m_{\mathcal E_0/\overline{\mathfrak M}_{L_r^Y}({\bf c}_0) 
\times Y/\overline{\mathfrak M}_{L_r^Y}({\bf c}_0)}
$$
be the relative Quot-scheme. If $\overline{\mathfrak M}_{L_r^Y}({\bf c}_m)
\ne \emptyset$, then $m$ is even and nonnegative. When $m \ge 0$, 
the moduli space $\overline{\mathfrak M}_{L}({\bf c}_{2m})$ is isomorphic to 
$\Quot^m_{\mathcal E_0/}$.
 
\medskip\noindent
{\bf Theorem B.}  
{\it Let $Y \subset \mathbb P^1\times \mathbb P^1\times \mathbb P^2$ be
a generic smooth Calabi-Yau hypersurface.
Let $\ep_1, \ep_2 = 0, 1$, and $k=(1+\ep_1)(4-\ep_2)(3-\ep_2)/2-1$.
Let $\pi: Y \to \mathbb P^1\times \mathbb P^2$ be the restriction 
to $Y$ of the projection of $\mathbb P^1\times \mathbb P^1 \times 
\mathbb P^2$ to the product of the last two factors.
Fix a point $y_0 \in Y$, and define in $A^*(Y)$ the class
\begin{eqnarray}   \label{B.1}
{\bf c}_m = -m[y_0] + \big ( 1+\pi^*(-1, 1) \big ) \cdot 
   \big ( 1+\pi^*(\ep_1+1, \ep_2-1) \big )
\end{eqnarray}
where $(a, b)$ denotes the divisor $a(\{p\} \times \Pee^2) 
+ b(\Pee^1 \times H)$ for a line $H$ in $\mathbb P^2$. 
\begin{enumerate}
\item[{\rm (i)}] If $0 < r < 2(2-\ep_2)/(2+\ep_1)$,
then $\lambda(L_r^Y, {\bf c}_{m}) = 0$ for all $m \in \Z$.

\item[{\rm (ii)}] If $2(2-\ep_2)/(2+\ep_1) < r < 2(2-\ep_2)/\ep_1$, then
\begin{eqnarray*}      
  \sum_{m \in \Z} \lambda(L_r^Y, {\bf c}_m) q^m
= (-1)^{k} \cdot (k+1) \cdot M(q^2)^{2 \, \chi(Y)}.
\end{eqnarray*}
\end{enumerate}}

In fact, if $0 < r < 2(2-\ep_2)/(2+\ep_1)$, then the moduli space 
$\overline{\mathfrak M}_{L}({\bf c}_m)$ is empty for all $m \in \Z$.
Therefore $r = 2(2-\ep_2)/(2+\ep_1)$ may be regarded as a wall, and
the open intervals $(0, 2(2-\ep_2)/(2+\ep_1))$ and 
$(2(2-\ep_2)/(2+\ep_1), 2(2-\ep_2)/\ep_1)$ may be regarded as two chambers. 
From this point of view, Theorem~B~(ii) provides 
a wall-crossing formula for the Donaldson-Thomas invariants.
Due to their connections with Hall algebras, 
wall-crossing formulas for Donaldson-Thomas invariants have been 
investigated intensively in the past few years (see \cite{Joy, JS, KL, KS, Tod1} 
and the references there). A concrete wall-crossing formula was obtained 
in \cite{KL} under certain conditions which are not satisfied 
in our present situation. Our results might shed some light on 
the general properties of these wall-crossing formulas.

In addition to understanding the moduli spaces of Gieseker semistable sheaves,
another essential ingredient in the proofs of Theorem~A and 
Theorem~B is the following result concerning Behrend's 
$\nu$-function when a torus action exists.

\medskip\noindent
{\bf Theorem C.}  
{\it Assume that $\mathbb T = \C^*$ acts on a complex scheme $X$
which admits a symmetric obstruction theory compatible 
with the $\mathbb T$-action. Let $P \in X^{\mathbb T}$. Then,
\begin{enumerate}
\item[(i)] $X^{\mathbb T}$ admits a symmetric obstruction theory;

\item[(ii)] $\displaystyle{\nu_X(P) = (-1)^{\dim T_PX \, - \,
\dim T_P(X^{\mathbb T})} \cdot \nu_{X^{\mathbb T}}(P)}$, where 
$T_PX$ denotes the Zariski tangent space of $X$ at $P$
and $\nu_X$ denotes Behrend's $\nu$-function for $X$.
\end{enumerate}}

When $P \in X^{\mathbb T}$ is an isolated $\mathbb T$-fixed point,
Theorem~C~(ii) has been proved in \cite{BF}. It also follows easily
when $X$ is locally the critical scheme of a regular function $f$ on 
a smooth scheme $M$, i.e., $X=Z(df)$ locally. This is due to the fact that, 
in this case, $\nu_X(P)$ can be computed via the Euler characteristic 
of the Milnor fiber obtained from $f$. In \cite{JS, BG}, it is shown 
that the moduli spaces of Gieseker stable sheaves on a Calabi-Yau 
$3$-fold are locally critical schemes. Therefore the results from 
\cite{JS, BG} are sufficient for the computation of Donaldson-Thomas 
invariants. However, in \cite{MPT}, an example where a scheme $X$ 
admitting a symmetric obstruction theory is not locally critical 
has been constructed. Hence Theorem C could be useful to a general 
$X$ with just a symmetric obstruction theory. 

By the results in \cite{Beh}, the Donaldson-Thomas invariant
$\lambda(L, {\bf c}_m)$ coincides with the weighted Euler characteristic
$\W \chi \big ( \overline{\mathfrak M}_L({\bf c}_m) \big )$ of 
the moduli space 
$\overline{\mathfrak M}_L({\bf c}_m)$ (see (\ref{weighted}) for the 
definition of $\W \chi ( \cdot )$). By standard techniques, the 
computation of $\W \chi \big ( \overline{\mathfrak M}_L({\bf c}_m) \big )$
reduces to the relevant punctual Quot-schemes. It turns out that 
these punctual Quot-schemes admit ${\mathbb T}$-actions. 
The ${\mathbb T}$-fixed loci are the unions of certain products of
the punctual Hilbert schemes of $0$-dimensional closed subschemes on $Y$.
A combination of Theorem~C and the results in \cite{JLi, BF, LP} 
regarding punctual Hilbert schemes yield Theorem~A and Theorem~B.

In \cite{MNOP}, the Donaldson-Thomas invariants are defined via 
the moduli space of ideal sheaves $I_Z$ where the dimensions of 
the closed subschemes $Z$ are equal to one. Most of the computational
results in the literature are concentrated on this type of invariants. 
Via the Serre construction, curves on $Y$ correspond to rank-$2$ vector 
bundles on $Y$. With this point of view, Theorem~A is comparable to 
the Donaldson-Thomas invariants corresponding to super-rigid curves 
considered in \cite{BB}. An irreducible super-rigid curve $C$ in $Y$ 
has normal bundle $\mathcal O_C(-1) \oplus \mathcal O_C(-1)$ and thus 
can not deform; in our case, $E_{0, 1}$ and $E_{0,2}$ have no 
deformation either, and thus are similar to super-rigid curves. 
Theorem~B is comparable to the Donaldson-Thomas invariants corresponding
to the moduli spaces of ideal sheaves $I_Z$ 
where the topological invariant $[Z]$ of $Z$ is fixed in $H_2(Y,\Bbb Z)$ 
but the curve component of $Z$ has a positive dimensional moduli.

We remark that recently Stoppa \cite{Sto} and Toda \cite{Tod2} worked on D0-D6 states counting
which  is similar to this paper. For example, sheaves in \cite{Sto}  are isomorphic to the trivial vector bundle 
of some rank outside a finite set of points. Thus the invariants counted there are higher rank generalizations of
degree zero Donaldson-Thomas invariants for ideal sheaves defined in \cite{MNOP}. While their work deal
with sheaves with vanishing  first and second Chern classes on arbitrary Calabi-Yau three-folds, 
sheaves in this paper have non-zero first and second Chern classes on some special Calabi-Yau
three-folds. So our paper studies the generalization of the Donaldson-Thomas invariants
for ideal sheaves in \cite{MNOP} with non-trivial contributions from  curve components. 
 While the methods used in \cite{Sto} and \cite{Tod2} include powerful wall-crossing techniques developed
in \cite{JS, KS} and Bridgeland stability conditions, this paper uses  complete descriptions of moduli spaces
to carry  out the computations.

The paper is organized as follows. In \S\ref{sect_Behrend}, we prove 
Theorem~C. In \S\ref{sect_Euler}, we review virtual 
Hodge polynomials, and compute the Euler characteristics of Grothendieck 
Quot-schemes. The results are of independent interest, and will be used 
in \S \ref{sect_Donaldson}. In \S \ref{sect_Donaldson}, we verify 
Theorem~A and Theorem~B.

\medskip\noindent
{\bf Acknowledgment}: The authors thank   H.L. Chang, Y.F. Jiang, S. Katz,   Jun Li, R. Pandharipande  for 
valuable discussions. In addition, the second author thanks HKUST for 
its warm hospitality and support during his visit in January 2010.

\section{\bf Behrend's $\nu$-functions for schemes with $\C^*$-actions}
\label{sect_Behrend}

For a complex scheme $X$, an invariant $\nu_X$ of $X$ was
introduced in \cite{Beh}. The invariant $\nu_X$ is 
an integer-valued constructible function defined over $X$.
Following \cite{Beh}, {\it the weighted Euler characteristic} of $X$
is defined to be
\begin{eqnarray}  \label{weighted}
\W \chi(X) = \chi(X, \nu_X) := \sum_{n \in \Z} n \cdot 
\chi \big ( \{x \in X |\nu_X(x) = n \} \big )
\end{eqnarray}
where $\chi(\cdot)$ denotes the usual Euler characteristic of topological spaces.
Let $\mathbb T = \C^*$ act on $X$ which admits a symmetric obstruction 
theory compatible with the $\mathbb T$-action. By \cite{BF}, 
if $P \in X^{\mathbb T}$ is an isolated $\mathbb T$-fixed point 
of $X$, then
\begin{eqnarray}  \label{isolated}
\nu_X(P) = (-1)^{\dim T_PX}
\end{eqnarray}
where $T_PX$ denotes the Zariski tangent space of $X$ at $P$. However, 
for the cases considered in our paper, the fixed point sets $X^{\mathbb T}$ 
will be positive dimensional. Our goal is to prove Theorem~C which 
generalizes (\ref{isolated}) to the case when $P \in X^{\mathbb T}$ is 
not necessarily an isolated $\mathbb T$-fixed point of $X$. Theorem~C reduces 
the computation of the $\nu$-function of $X$  to that of the fixed point set
$X^{\mathbb T}$.

We begin with a few technical lemmas regarding Thom classes. 
Elementary properties about Thom classes can be found in \cite{BT}. Let 
\begin{eqnarray*}
\Delta_n &=& \{(z_1, \ldots, z_n; \overline z_1, \ldots, \overline z_n)\} 
      \subset \mathbb C^n\times \mathbb C^n,   \\
F &=& \{(z_1, \ldots, z_m, 0, \ldots, 0; {\w z}_1, \ldots, {\w z}_m, 0, 
      \ldots, 0)\} \subset \mathbb C^n\times \mathbb C^n.
\end{eqnarray*}
We orient $\Delta_n$ so that the natural map $\C^n \to \Delta_n$ defined by
\begin{eqnarray}  \label{ori-Delta-n}
(z_1, \ldots, z_n) \mapsto (z_1, \ldots, z_n; \overline z_1, \ldots, 
\overline z_n)
\end{eqnarray}
is orientation-preserving. Let $\Delta_m = \Delta_n \cap F$. 
The intersection of $\Delta_n$ and $F$
along $\Delta_m$ is not transversal. However, a direct computation shows 
that $(T_{\Delta_n}|_{\Delta_m}) \cap (T_F|_{\Delta_m}) = T_{\Delta_m}$.
Consider the following commutative diagram of maps:
\begin{eqnarray}   \label{BigDiagram}
\begin{array}{ccccccc}
&0&&0&&0&   \\
&\uparrow&&\uparrow&&\uparrow&     \\
0\to&N_{\Delta_m\subset F}&\to &N_{\Delta_n \subset \mathbb C^n\times 
  \mathbb C^n}|_{\Delta_m}&\to &M&\to 0       \\
&\uparrow&&\uparrow&&\uparrow&        \\
0\to&T_F|_{\Delta_m}&\to& T_{\mathbb C^n\times \mathbb C^n}|_{\Delta_m}&\to 
  &N_{F\subset \mathbb C^n\times \mathbb C^n}|_{\Delta_m}&\to 0    \\
&\uparrow&&\uparrow&&\uparrow&      \\
0\to&T_{\Delta_m}&\to &T_{\Delta_n}|_{\Delta_m}&\to&N_{\Delta_m\subset 
  \Delta_n}&\to 0      \\
&\uparrow&&\uparrow&&\uparrow&       \\
&0&&0&&0&    
\end{array}
\end{eqnarray}
where $M$ is defined to be the quotient $\big (N_{\Delta_n \subset \mathbb C^n
\times \mathbb C^n}|_{\Delta_m} \big )/N_{\Delta_m\subset F}$. Note that
$M$ is a rank-$(2n-2m)$ real vector bundle over $\Delta_m$.
Let $\omega_{\Delta_n}$, $\omega_{\Delta_m}$ and $\omega_M$ be the Thom classes 
of the vector bundles $N_{\Delta_n \subset \mathbb C^n\times\mathbb C^n}$,
$N_{\Delta_m\subset F}$ and $M$ respectively.

\begin{lemma}  \label{excess}
Let $i: F \hookrightarrow \mathbb C^n\times \mathbb C^n$ be the inclusion map,
let $\tau: N_{\Delta_m\subset F} \to \Delta_m$ be the natural projection,
and let $\chi(M)$ be the Euler class of $M$. Then,
$$
i^*\omega_{\Delta_n} = \omega_{\Delta_m} \cup \tau^*\chi(M).
$$
\end{lemma}
\begin{proof}
The exact sequence $0\to N_{\Delta_m\subset F}\to N_{\Delta_n \subset \mathbb C^n
\times \mathbb C^n}|_{\Delta_m}\to M\to 0$ splits as smooth vector bundles.
Let $\pi_1$ and $\pi_2$ be the projections of the bundle $N_{\Delta_m\subset F}
\oplus M$ to $N_{\Delta_m\subset F}$ and $M$ respectively. 
Then we have a commutative diagram:
\begin{eqnarray*}
\begin{array}{cccc}
N_{\Delta_n \subset \mathbb C^n\times\mathbb C^n}|_{\Delta_m}\cong&
   N_{\Delta_m\subset F}\oplus M&\buildrel{\pi_2}\over\longrightarrow &M  \\
&\downarrow{\pi_1}&&\downarrow{}  \\
&N_{\Delta_m\subset F}&\buildrel\tau\over\longrightarrow &\Delta_m.
\end{array}
\end{eqnarray*}
Thus $\omega_{\Delta_n}|_{(N_{\Delta_n \subset \C^n \times \C^n}|_{\Delta_m})}
=\pi_1^*\omega_{\Delta_m} \cup \pi_2^*\omega_M$ where $\pi_2^*\omega_M$ is 
the Thom class of $\tau^*M$ over the space $N_{\Delta_m\subset F}$. 
Since $N_{\Delta_m\subset F}\oplus M$ can be identified with $\tau^*M$ and 
$$
N_{\Delta_m\subset F}\hookrightarrow N_{\Delta_m\subset F}\oplus M=\tau^*M
$$ 
can be regarded as the zero section of the bundle $\tau^*M$ over 
$N_{\Delta_m\subset F}$, we have
$$
\omega_{\Delta_n}|_{(N_{\Delta_m\subset F})}
=\omega_{\Delta_m} \cup \tau^*\chi(M).
$$
Here we have used the fact that if $s$ is the zero section of a vector bundle 
$E\to Y$ and $\omega$ is the Thom class of $E$, then $s^*\omega$ is the Euler 
class $\chi(E)$ of $E$. Since $\Delta_m$ is a real linear subspace of $F$,
$N_{\Delta_m \subset F} = F$. Therefore, $i^*\omega_{\Delta_n} = 
\omega_{\Delta_m} \cup \tau^*\chi(M)$.
\end{proof}

Next, we fix some $S^1$-action on $\mathbb C^n\times \mathbb C^n$.
Let $z_1, \ldots, z_n, {\w z}_1, \ldots, {\w z}_n$ 
be the coordinates of $\C^n \times \C^n$. 
For $t \in S^1$ and $1 \le i \le n$, let $t(z_i) = t^{r_i} z_i$ 
and $t({\w z}_i) = t^{-r_i} {\w z}_i$ where $r_i \in \Z$. Let
$r_i = 0$ for $1 \le i \le m$ and $r_i \ne 0$ for $m+1 \le i \le n$. Then, 
$$
F = (\mathbb C^n\times \mathbb C^n)^{S^1}.
$$

\begin{lemma}  \label{chiM}
In the $S^1$-equivariant cohomology $H^*_{S^1}(F)$, we have
$$
i^*\omega_{\Delta_n} 
= (-1)^{n-m}(r_{m+1} \cdots r_n)t^{n-m} \,\, \omega_{\Delta_m}.
$$
\end{lemma}
\begin{proof}
Note that the arguments in the proof of Lemma~\ref{excess} go through
in the $S^1$-equivariant setting (e.g., the smooth splitting 
$N_{\Delta_n \subset \mathbb C^n\times\mathbb C^n}|_{\Delta_m} \cong 
N_{\Delta_m\subset F}\oplus M$ is $S^1$-equivariant). 
Thus Lemma~\ref{excess} holds in $H^*_{S^1}(F)$, and it suffices to show that
\begin{eqnarray}     \label{chiM.1}
\chi_{S^1}(M) = (-1)^{n-m}(r_{m+1} \cdots r_n)t^{n-m}.
\end{eqnarray}
Let $P \in \mathbb C^n\times \mathbb C^n$ be the origin. 
By the first horizontal exact sequence in (\ref{BigDiagram}),
\begin{eqnarray}     \label{chiM.2}
M|_P =
(N_{\Delta_n \subset \C^n\times \C^n})|_P/(N_{\Delta_m\subset F})|_P.
\end{eqnarray}

Let $N_n = (N_{\Delta_n \subset \C^n\times \C^n})|_P = 
T_P(\mathbb C^n\times \mathbb C^n)/T_P\Delta_n$ be the normal space.
For each $z_i\in \mathbb C$, write $z_i=x_i+\sqrt{-1}y_i$ and 
$\w z_i= \w x_i+\sqrt{-1} \w y_i$. Then,
$$
T_P\Delta_n = \{(z_1, \ldots, z_n; \overline z_1, \ldots, \overline z_n)\}
\subset T_P(\mathbb C^n\times \mathbb C^n).
$$
As an $\mathbb R$-subspace of 
$T_P(\mathbb C^n\times \mathbb C^n) \cong \C^n\times \C^n$,
$T_P\Delta_n$ has a basis consisting of
\begin{eqnarray*}
\vec u_1&=&(1, 0, \ldots, 0, 0; 1, 0, \ldots, 0, 0),\\
\vec u_2&=&(0, 1, \ldots, 0,0; 0, -1, \ldots, 0, 0),\\
&\vdots&\\
\vec u_{2n-1}&=&(0, 0, \ldots, 1, 0;0, 0,\ldots, 1, 0),\\
\vec u_{2n}&=&(0, 0, \ldots, 0, 1;0, 0, \ldots, 0, -1).
\end{eqnarray*}
Recall that the orientation of $\Delta_n$ is such that the map 
(\ref{ori-Delta-n}) is orientation preserving. Thus the ordered basis 
$\{ \vec u_1, \ldots, \vec u_{2n}\}$ is the orientation of $T_P\Delta_n$. 
Define 
$$
\varphi\colon T_P(\mathbb C^n\times \mathbb C^n ) \to \mathbb C^n
$$ 
by $\varphi(z_1, \ldots, z_n; \w z_1, \ldots, \w z_n) = 
(\w z_1-\overline z_1, \ldots, \w z_n-\overline z_n)$.
It is easy to check that $\varphi(T_P\Delta_n)=0$. 
Thus $\varphi$ induces an isomorphism of vector spaces over $\mathbb R$:
$$
N_n=\frac{T_P(\mathbb C^n\times \mathbb C^n)}{T_P\Delta_n}\to \mathbb C^n,
$$
stilled denoted by $\varphi$. The quotient space $N_n$ is equipped with 
the orientation 
\begin{eqnarray*}
\vec v_1&=&(0, \ldots, 0; 1, \ldots, 0),\\
&\vdots&\\
\vec v_{2n}&=&(0, \ldots, 0; 0, \ldots, 1).
\end{eqnarray*}
Now $\varphi(\vec v_1)=(1, \ldots, 0), \ldots, \varphi(\vec v_{2n}) =
(0, \ldots, 1)$. Since $\{\vec u_1, \ldots, \vec u_{2n}, \vec v_1, \ldots, 
\vec v_{2n}\}$ agrees with the orientation of $T_P(\C^n\times \C^n)$, 
$\varphi$ is orientation-preserving. 

Similarly, let $N_m = (N_{\Delta_m\subset F})|_P = T_PF/T_P\Delta_m$. 
Then we have
\begin{eqnarray*}
N_m = 
\frac{\{(z_1, \ldots, z_m, 0,\ldots, 0;\w z_1, \ldots,\w z_m, 0\ldots, 0)\}}
   {\{(z_1, \ldots, z_m, 0, \ldots, 0;
   \overline z_1, \ldots, \overline z_m, 0, \ldots, 0)\}}
\subset \frac{T_P(\mathbb C^n\times \mathbb C^n)}{T_P\Delta_n} = N_n.
\end{eqnarray*}
Introduce a well-defined map $\psi\colon N_n\to \mathbb C^{n-m}$ given by
$$
[(z_1, \ldots, z_n; \w z_1, \ldots, \w z_n)] \mapsto 
(\w z_{m+1}-\overline z_{m+1}, \ldots, \w z_n-\overline z_n).
$$
Clearly  $\psi(N_m)=0$. Therefore $\psi$ induces an isomorphism of vector
spaces over $\mathbb R$. 
$$
M|_P =\frac{N_n}{N_m}\to \mathbb C^{n-m},
$$
stilled denoted by $\psi$, where we have used (\ref{chiM.2}). One checks that
\begin{eqnarray*}
\psi(\vec v_{2m+1})=(1, \ldots, 0), \ldots, \psi(\vec v_{2n})=(0, \ldots, 1).
\end{eqnarray*}
So $\psi$ is an orientation-preserving isomorphism. Define an $S^1$-action 
on $\C^{n-m}$ by 
$$
t\cdot (a_{m+1}, \ldots, a_n)=(t^{-r_{m+1}}a_{m+1}, \ldots, t^{-r_n}a_n).
$$
Then the isomorphism $\psi$ is $S^1$-equivariant since
\begin{eqnarray*}
&&\psi\big(t\cdot [(z_1, \ldots, z_n; \w z_1, \ldots, \w z_n)]\big)\\
&=&\psi\big([(z_1, \ldots, z_m, t^{r_{m+1}}z_{m+1}, \ldots, t^{r_n}z_n; 
  \w z_1, \ldots, \w z_m, t^{-r_{m+1}}\w z_{m+1}, \ldots, t^{-r_n}\w z_n)]\big)\\
&=&(t^{-r_{m+1}}\w z_{m+1}-\overline{t^{r_{m+1}}z_{m+1}}, \ldots, 
  t^{-r_n}\w z_n -\overline{t^{r_n}z_n})\\
&=&\big(t^{-r_{m+1}}(\w z_{m+1}-\overline z_{m+1}), \ldots, 
  t^{-r_n}(\w z_n-\overline z_n)\big)\\
&=&t\cdot \psi\big([(z_1, \ldots, z_n; \w z_1, \ldots, \w z_n)]\big).
\end{eqnarray*}
It follows that 
$\chi_{S^1}(M)=(-r_{m+1})t \cdots (-r_n)t=(-1)^{n-m}(r_{m+1}\cdots r_n)t^{n-m}$.
\end{proof}

We continue with the $S^1$-action on $\mathbb C^n\times \mathbb C^n$ defined earlier.
Fix holomorphic functions $f_1, \ldots, f_n$ in the variables $z_1, \ldots, z_n$
such that the degree of each $f_i$ with respect to the $S^1$-action is $-r_i$. 
Then $f_i(z_1, \ldots, z_m, 0, \ldots, 0) = 0$ for all $m+1 \le i \le n$. For $1 \le 
i \le m$, let $\w f_i(z_1, \ldots, z_m) = f_i(z_1, \ldots, z_m, 0, \ldots, 0)$.
Then, $\w f_1, \ldots, \w f_m$ are holomorphic in $z_1, \ldots, z_m$.
Let $\Gamma \subset \C^n \times \C^n$ be defined by the equations $\w z_1 = f_1, 
\ldots, \w z_n = f_n$. Then $\Gamma \cap F$ is given by 
$\{\w z_1 = \w f_1, \ldots, \w z_m = \w f_m \} \subset F$.

\begin{lemma}  \label{Gamma}
In the $S^1$-equivariant cohomology $H^*_{S^1}(F)$, we have
$$
i^*\omega_{\Gamma} 
= (-1)^{n-m}(r_{m+1} \cdots r_n)t^{n-m} \,\, \omega_{\Gamma \cap F},
$$
where $\omega_\Gamma$ is the Thom class of the normal bundle of $\Gamma$ in 
$\mathbb C^n\times \mathbb C^n$. and $\omega_{\Gamma\cap F}$ is the Thom class of the normal
bundle of $\Gamma\cap F$ in $F$.
\end{lemma}
\begin{proof}
First of all, note that the tangent space of $\Gamma$ is spanned by 
the ordered basis
$$
(1, \ldots, 0; \frac{\partial f_1}{\partial z_1}, \ldots, 
\frac{\partial f_n}{\partial z_1}), \ldots, (0, \ldots, 1; 
\frac{\partial f_1}{\partial z_n}, \ldots, 
\frac{\partial f_n}{\partial z_n}).
$$
Thus the normal vector bundle $N_{\Gamma\subset \mathbb C^n\times \mathbb C^n}$ 
is spanned by 
$$
\big(\frac{\partial f_1}{\partial z_1},\ldots, \frac{\partial f_1}{\partial z_n}; 
-1, \ldots 0\big), \ldots, \big(\frac{\partial f_n}{\partial z_1}, \ldots, 
\frac{\partial f_n}{\partial z_n}; 0, \ldots, -1\big).
$$
Using these, one checks that $(T_{\Gamma}|_{\Gamma \cap F}) \cap 
(T_F|_{\Gamma \cap F}) = T_{\Gamma \cap F}$ and that 
$$
\W M \,\, := \,\,\big (N_{\Gamma \subset \C^n \times 
\C^n}|_{\Gamma \cap F} \big )/N_{\Gamma \cap F\subset F}
$$
is a complex bundle. By the same arguments as in the proofs of 
Lemma~\ref{excess} and Lemma~\ref{chiM}, 
$i^*\omega_{\Gamma} = \omega_{\Gamma \cap F} \cup \chi_{S^1}(\W M)$. 
Hence it remains to prove that 
$$
\chi_{S^1}(\W M) = (-1)^{n-m}(r_{m+1}\cdots r_n)t^{n-m}.
$$

Next, note from the assumptions about the functions $f_1, \ldots, f_m, f_{m+1}, 
\ldots, f_n$ that the restriction $N_{\Gamma\subset \C^n \times \C^n}|_{\Gamma\cap F}$ 
is spanned by
\begin{eqnarray*}
 \vec w_1&=&\big(\frac{\partial \w f_1}{\partial z_1}, \ldots, 
\frac{\partial \w f_1}{\partial z_m}, 0. \ldots, 0; -1, \ldots, 0, 0\ldots, 0\big),\\
 &\vdots\\
 \vec w_m&=&\big(\frac{\partial \w f_m}{\partial z_1}, \ldots, 
\frac{\partial \w f_m}{\partial z_m}, 0. \ldots, 0; 0, \ldots, -1, 0\ldots, 0\big),\\
 \vec w_{m+1}&=&\big( 0. \ldots, 0, \frac{\partial f_{m+1}}{\partial z_{m+1}}, 
\ldots, \frac{\partial f_{m+1}}{\partial z_n}; 0, \ldots, 0, -1\ldots, 0\big),\\
 &\vdots&\\
  \vec w_{n}&=&\big( 0. \ldots, 0, \frac{\partial f_{n}}{\partial z_{m+1}}, \ldots, 
\frac{\partial f_{n}}{\partial z_n}; 0, \ldots, 0, 0\ldots, -1\big).
\end{eqnarray*}
In addition, since the degree of $f_{m+1}$ with respect to the $S^1$-action is $-r_{m+1}$,
for each component $\partial f_{m+1}/\partial z_j$ in $\vec w_{m+1}$, we have
either $\partial f_{m+1}/\partial z_j = 0$ or $r_j=-r_{m+1}$. So $t\cdot \vec w_{m+1} = 
t^{-r_{m+1}} \vec w_{m+1}$ under the $S^1$-action on $\C^n \times \C^n$, where we regard
$\vec w_{m+1} \in \C^n \times \C^n$ with $z_{m+1} = \ldots = z_n = 0$ and with 
$z_1, \ldots, z_m$ fixed. Similarly,
$$
t\cdot \vec w_j=t^{-r_j}\vec w_j\quad\hbox{ for }\quad j\ge m+1.
$$

Since the normal bundle $N_{\Gamma\cap F\subset F}$ is spanned by 
$\vec w_1, \ldots, \vec w_m$, the bundle $\W M$ is spanned by 
$\vec w_{m+1}, \ldots, \vec w_n$, i.e., the bundle $\W M$ is $S^1$-equivariantly 
isomorphic to the trivial bundle $(\Gamma\cap F)\times \mathbb C^{n-m}$ where 
$S^1$ acts on $\mathbb C^{n-m}$ by
$$
t\cdot (b_{m+1}, \ldots, b_n)=(t^{-r_{m+1}}b_{m+1}, \ldots, t^{-r_n}b_n).
$$
Therefore, $\chi_{S^1}(\W M) = (-r_{m+1})t \cdots (-r_n)t = 
(-1)^{n-m}(r_{m+1}\cdots r_n)t^{n-m}$.
\end{proof}

\begin{theorem}  \label{fixedset}
Assume that $\mathbb T$ acts on a complex scheme $X$
and that $X$ admits a symmetric obstruction theory compatible 
with the $\mathbb T$-action. Let $P \in X^{\mathbb T}$. Then,
\begin{enumerate}
\item[(i)] $X^{\mathbb T}$ admits a symmetric obstruction theory;

\item[(ii)] $\displaystyle{\nu_X(P) = (-1)^{\dim T_PX \, - \,
\dim T_P(X^{\mathbb T})} \cdot \nu_{X^{\mathbb T}}(P)}$.
\end{enumerate}
\end{theorem}
\begin{proof}
We will prove the statements \' etale locally. 
As in Subsect.~3.2 of \cite{BF}, we may assume that
$X = Z(\omega) \subset \C^n$ and $P \in X$ is the origin of $\C^n$.
Here $n = \dim T_P X$ and $\omega$ is a $\mathbb T$-invariant 
almost closed $1$-form on $\C^n$. Let $z_1, \ldots, z_n$ be
the coordinates of $\C^n$, and $z_1, \ldots, z_n, {\w z}_1, \ldots, {\w z}_n$ 
be the coordinates of $\Omega_{\C^n} = \C^n \times \C^n$. 
Let $\omega = \sum_{i=1}^n f_i dz_i$
where the functions $f_i$ are holomorphic in $z_1, \ldots, z_n$. 

(i) For $t \in \mathbb T$ and $1 \le i \le n$, let $t(z_i) = t^{r_i} z_i$ 
where $r_i \in \Z$. Then the degrees of both ${\w z}_i$ and $f_i$ 
with respect to the $\mathbb T$-action are equal to $-r_i$. Assume that 
$r_i = 0$ for $1 \le i \le m$ and $r_i \ne 0$ for $m+1 \le i \le n$. 
Then $f_i(z_1, \ldots, z_m, 0, \ldots, 0) = 0$ for all $m+1 \le i \le n$. For $1 \le 
i \le m$, let $\w f_i(z_1, \ldots, z_m) = f_i(z_1, \ldots, z_m, 0, \ldots, 0)$.

Since $X = Z(\omega) = \{f_1=0, \ldots, f_n=0\}$ and $(\C^n)^{\mathbb T} 
= \{z_{m+1} = 0, \ldots, z_n=0\}$, 
\begin{eqnarray*}
   X^{\mathbb T} = X \cap (\C^n)^{\mathbb T}
&=&\{f_1=0, \ldots, \,\, f_n=0, \,\, z_{m+1} = 0, \ldots, \,\, z_n=0\} \\
&=&\{\w f_1 =0, \ldots, \,\, \w f_m =0, \,\, z_{m+1} = 0, 
    \ldots, \,\, z_n=0\}.
\end{eqnarray*}
So regarded as a subvariety of $\C^m$, $X^{\mathbb T} = Z(\omega^{\mathbb T})$ 
where $\omega^{\mathbb T} = \sum_{i=1}^m \w f_i dz_i$.

To show that $X^{\mathbb T}$ admits a symmetric obstruction theory,
it suffices to prove that
\begin{eqnarray}  \label{fixedset.1}
{\partial \w f_i \over \partial z_j} \equiv 
{\partial \w f_j \over \partial z_i}
\pmod {(\w f_1, \ldots, \w f_m)}
\end{eqnarray}
for $1 \le i, j \le m$. Since $\omega = \sum_{i=1}^n f_i dz_i$ is 
an almost closed $1$-form, we have
\begin{eqnarray*}  
{\partial f_{i} \over \partial z_j} \equiv 
{\partial f_{j} \over \partial z_i}
\pmod {(f_{1}, \ldots, f_{n})},
\end{eqnarray*}
for $1 \le i, j \le m$, i.e., there exist holomorphic functions 
$g_1, \ldots, g_n$ such that
\begin{eqnarray}  \label{fixedset.2}
{\partial f_i \over \partial z_j} = 
{\partial f_j \over \partial z_i} + \sum_{k=1}^n g_k f_k.
\end{eqnarray}
Setting $z_{m+1} = \ldots = z_n =0$, we obtain the relations
(\ref{fixedset.1}).

(ii) First of all, we claim that $\dim T_P(X^{\mathbb T}) = m$.
Indeed, let $m_P^\mathbb T$ be the maximal ideal of the local ring
of $X^{\mathbb T}$ at $P$. Since $X^{\mathbb T}$ is 
defined by $\w f_1 =0, \ldots, \w f_m =0$, 
\begin{eqnarray}    \label{fixedset.3}
\dim T_P(X^{\mathbb T})
= \dim \, {m_P^\mathbb T \over \big ( m_P^\mathbb T \big )^2} 
= m - \rank \big ( J^\mathbb T \big )
\end{eqnarray}
where $J^\mathbb T = \big [ \partial \w f_i/\partial z_j(P) 
\big ]_{1 \le i, j \le m}$ is the Jacobian matrix at 
$P = (0, \ldots, 0) \in X^{\mathbb T} \subset \C^m$. 
Similarly, since $X$ is defined by the equations $f_1=0, \ldots, f_n=0$,
\begin{eqnarray*} 
\dim T_P X = n - \rank (J)
\end{eqnarray*}
where $J = \big [ \partial f_{i}/\partial z_j(P) \big ]_{1 \le i, 
j \le n}$ is the Jacobian matrix at $P = (0, \ldots, 0) \in 
X \subset \C^n$. Note that $\dim T_P X = n$. So $\rank (J) = 0$,
and $\partial f_i/\partial z_j(P) = 0$ for all $1 \le i, j \le m$. 
Setting $z_{m+1} = \ldots = z_n = 0$, we conclude that 
$\partial \w f_i/\partial z_j(P) = 0$ for all $1 \le i, j \le m$.
Hence $J^\mathbb T$ is the zero matrix, and 
$\dim T_P(X^{\mathbb T}) = m$ in view of (\ref{fixedset.3}).

Next, let $\mathcal C \hookrightarrow \Omega_{\C^n}$ be the embedding 
of the normal cone $\mathcal C_{X/\C^n}$ into $\Omega_{\C^n}$ given by 
$\omega$. Let $\Delta_n$ be the subspace of $\Omega_{\C^n} = \C^n
\times \C^n$ consisting of all the points $(z_1, \ldots, z_n, 
\overline{z}_1, \ldots, \overline{z}_n)$. Orient $\Delta_n$ so that 
the map $\C^n \to \Delta_n$ is orientation-preserving. For $\eta \in 
\C -\{0\}$, let $\Gamma_\eta$ be the graph of the section $1/\eta \cdot 
\omega$ of $\Omega_{\C^n}$, i.e., $\Gamma_\eta$ is the subspace of 
$\Omega_{\C^n}$ defined by the equations $\eta {\w z}_i = f_i$, 
$1 \le i \le n$. Again, orient $\Gamma_\eta$ so that the map 
$\C^n \to \Gamma_\eta$ is orientation-preserving. From the proof of
Proposition~4.22 in \cite{Beh}, we see that $P$ is an isolated point 
of the intersection $\mathcal C \cap \Delta_n$, $\lim_{\eta \to 0}
[\Gamma_\eta] = [\mathcal C]$, and $\nu_X(P) = I_{\{P\}}([\mathcal C], 
[\Delta_n])$ which denotes the intersection number at $P$ of the cycles
$[\mathcal C]$ and $[\Delta_n]$. Therefore, $\nu_X(P) = I_{\{P\}}
([\Gamma_\eta], [\Delta_n])$ whenever $|\eta| \ne 0$ is sufficiently 
small. For simplicity, let $\eta = 1$ and $\Gamma = \Gamma_1$. 
Then 
\begin{eqnarray*}  
\nu_X(P) = I_{\{P\}}([\Gamma], [\Delta_n]).
\end{eqnarray*}
Computing this in the equivariant cohomology $H^*_{S^1}(\Omega_{\C^n})$
where the $S^1$-action on $\Omega_{\C^n}$ is the one induced from the 
$\mathbb T$-action on $\Omega_{\C^n}$, we obtain
\begin{eqnarray}  \label{fixedset.4}
\nu_X(P) = \int_{\Omega_{\C^n}} \omega_\Gamma \cup \omega_{\Delta_n}
\end{eqnarray}
where $\omega_\Gamma \cup \omega_{\Delta_n} \in H^*_{S^1}(\Omega_{\C^n})$. 
Similarly, regarding $\Omega_{\C^m} \subset \Omega_{\C^n}$ and
letting $\Gamma^{\mathbb T} = \Gamma \cap \Omega_{\C^m}$ be defined 
by the equations ${\w z}_i = \w f_i$, $1 \le i \le m$, we have
\begin{eqnarray}  \label{fixedset.5}
\nu_{X^{\mathbb T}}(P) = \int_{\Omega_{\C^m}} \omega_{\Gamma^{\mathbb T}} 
\cup \omega_{\Delta_m}
\end{eqnarray}

Let $i: F := \Omega_{\C^m} \hookrightarrow 
\Omega_{\C^n}$ be the inclusion map. Using the localization theorem and 
computing in the localized equivariant cohomology, we have
\begin{eqnarray*}  
   \omega_\Gamma \cup \omega_{\Delta_n}
= i_* \left ( {i^*(\omega_\Gamma \cup \omega_{\Delta_n}) \over 
     \chi_{S^1}(N_{\Omega_{\C^m} \subset \Omega_{\C^n}})} \right )   
= i_* \left ( {i^*\omega_\Gamma \cup i^*\omega_{\Delta_n} \over 
     (-1)^{n-m}(r_{m+1}^2 \cdots r_n^2)t^{2n-2m}} \right ).
\end{eqnarray*}
Therefore, we conclude from Lemma~\ref{chiM} and Lemma~\ref{Gamma} that
\begin{eqnarray*}  
\omega_\Gamma \cup \omega_{\Delta_n} = (-1)^{n-m} \cdot 
i_*(\omega_{\Gamma^{\mathbb T}} \cup \omega_{\Delta_m}).    
\end{eqnarray*}
In view of (\ref{fixedset.4}) and (\ref{fixedset.5}), we obtain 
$\nu_X(P) = (-1)^{n-m} \cdot \nu_{X^{\mathbb T}}(P)$.
\end{proof}

\section{\bf Euler characteristics of Grothendieck Quot-schemes}
\label{sect_Euler}

In this section, we will compute the Euler characteristics of 
the Quot scheme $\Quot^m_{\mathcal O^{\oplus r}_Y}$ where $Y$ 
is a smooth projective variety. The results will be used in the 
next section for the computation of Behrend's $\nu$-function. 
\subsection{\bf Virtual Hodge polynomials and Euler characteristics}
\label{subsect_Virtual}
$\,$

First of all, let $Y$ be a reduced complex scheme 
(not necessarily projective, irreducible or smooth). 
Mixed Hodge structures are defined on the cohomology 
$H_c^k(Y, \mathbb Q)$ with compact support 
(see \cite{Del, DK}). The mixed Hodge structures coincide 
with the classical one if $Y$ is projective and smooth. 
For each pair of integers $(m, n)$, 
define the virtual Hodge number
$$e^{m, n}(Y) = \sum_k (-1)^k h^{m, n}(H_c^k(Y, \mathbb Q)).$$ 
Then the {\it virtual Hodge polynomial} of $Y$ is 
defined to be
\begin{eqnarray}   \label{e_red}
e(Y; s, t) = \sum_{m, n} e^{m, n}(Y) s^m t^n.
\end{eqnarray}

Next, for an arbitrary complex scheme $Y$, we put 
\begin{eqnarray}   \label{e}
e(Y; s, t) = e(Y_{\text{red}}; s, t)
\end{eqnarray}
following \cite{Che}. By (\ref{e}) and the results 
in \cite{DK, Ful, Che} for reduced complex schemes, we see that 
virtual Hodge polynomials satisfy the following properties:

\begin{enumerate}
\item[(i)] When $Y$ is projective and smooth, $e(Y; s, t)$ is 
the usual Hodge polynomial of $Y$. 
For a general complex scheme $Y$, we have
\begin{eqnarray}   \label{hodge_i}
e(Y; 1, 1)= \chi(Y).
\end{eqnarray}

\item[(ii)] If $\displaystyle{Y = \coprod_{i=1}^n Y_i}$ is 
a finite disjoint union of locally closed subsets, then
\begin{eqnarray}   \label{hodge_ii}
e(Y; s, t) = \sum_{i=1}^n e(Y_i; s, t).
\end{eqnarray}

\item[(iii)] If $f: Y \to Y'$ is a Zariski-locally trivial bundle 
with fiber $F$, then
\begin{eqnarray}   \label{hodge_iii}
e(Y; s, t) = e(Y'; s, t) \cdot e(F; s, t).
\end{eqnarray}

\item[(iv)] If $f: Y \to Y'$ is a bijective morphism, then 
\begin{eqnarray}   \label{hodge_iv}
e(Y; s, t) = e(Y'; s, t).
\end{eqnarray}
\end{enumerate}

Let $\mathbb T = \C^*$. By the Theorem 4.1 in \cite{LY}, 
if $Y$ admits a $\mathbb T$-action, then 
\begin{eqnarray}   \label{C*}
\chi(Y) = \chi \big (Y^{\mathbb T} \big ).
\end{eqnarray}

\subsection{\bf Euler characteristics of Grothendieck Quot-schemes}
\label{subsect_Euler}
$\,$

Let $Y$ be a projective scheme over a base Noetherian scheme $B$,
and let $\mathcal V$ be a sheaf on $Y$ flat over $B$. 
Let $\Quot^m_{\mathcal V/Y/B}$ be the (relative) Grothendieck Quot-scheme
parametrizing all the surjections $\mathcal V|_{Y_b} \to Q \to 0$,
modulo automorphisms of $Q$, with $b \in B$ such that 
the quotients $Q$ are torsion sheaves supported at finitely many
points and $h^0(Y_b, Q) = m$. When $B = {\rm Spec}(\C)$, we put
$\Quot^m_{\mathcal V} = \Quot^m_{\mathcal V/Y/{\rm Spec}(\C)}$.
An element in $\Quot^m_{\mathcal V}$ can also be regarded as 
a subsheaf $E \subset \mathcal V$ such that the quotient $\mathcal V/E$ is 
supported at finitely many points with $h^0(Y, \mathcal V/E) = m$.

By the Lemma~5.2~(ii) in \cite{LQ1}, if $\mathcal V$ is a locally free
rank-$r$ sheaf on $Y$, then
\begin{eqnarray}   \label{e_lf}
e \big ( \Quot^m_{\mathcal V/Y/B}; s, t \big )
= e \big ( \Quot^m_{\mathcal O_Y^{\oplus r}/Y/B}; s, t \big ).
\end{eqnarray}

In the rest of this section, let $Y$ be smooth. 
For a fixed point $y \in Y$, let
\begin{eqnarray*}   \label{punc_Q}
\Quot^m_{\mathcal O^{\oplus r}}(Y, y) \subset 
\Quot^m_{\mathcal O^{\oplus r}_Y}
\end{eqnarray*}
be the {\it punctual} Quot-scheme consisting of all the surjections 
$\mathcal O^{\oplus r}_Y \to Q \to 0$ such that
$\Supp(Q) = \{y\}$ and $h^0(Y, Q) = m$. Let $n = \dim Y$. Then,
\begin{eqnarray}   \label{punc_Q_C}
\Quot^m_{\mathcal O^{\oplus r}}(Y, y) \cong
\Quot^m_{\mathcal O^{\oplus r}}(\C^n, O) 
\end{eqnarray}
where $O$ denotes the origin in $\C^n$.
Also, when $r = 1$, $\Quot^m_{\mathcal O_Y}$ is the Hilbert scheme 
$\Hilb^m(Y)$, and $\Quot^m_{\mathcal O}(Y, y)$ is
the punctual Hilbert scheme $\Hilb^m(Y, y)$.

\begin{lemma}   \label{Q_pQ}
Let $Y$ be smooth, and let $O$ be the origin in $\C^n$. Then,
\begin{eqnarray}   \label{chi_Q}
\sum_{m=0}^{+\infty} \chi(\Quot^m_{\mathcal O^{\oplus r}_Y}) q^m 
= \left ( \sum_{m=0}^{+\infty} \chi \big 
  (\Quot^m_{\mathcal O^{\oplus r}}(\C^n, O) \big ) q^m \right )^{\chi(Y)}.
\end{eqnarray}
\end{lemma}
\begin{proof}
There exist unique rational numbers $Q_{n, r; k, \ell, m}$ such that
\begin{eqnarray}   \label{power_hil1}
\sum_{m=0}^{+\infty} e \big (\Quot^m_{\mathcal O^{\oplus r}}
  (\C^n, O); s, t \big ) q^m
= \prod_{k=1}^{+\infty} \,\, \prod_{\ell,m = 0}^{+\infty} \left ( 
  {1 \over {1 - q^k s^\ell t^m}} \right )^{Q_{n, r; k, \ell, m}}
\end{eqnarray}
as elements in $\mathbb Q[s, t][[q]]$. Define $\mathfrak Q_{n, r}(q,s,t)
\in \mathbb Q[s, t][[q]]$ to be the power series:
\begin{eqnarray}   \label{power_hil2}
\mathfrak Q_{n, r}(q, s,t)
= \sum_{k=1}^{+\infty} \left ( \sum_{\ell, m = 0}^{+\infty}
Q_{n, r; k, \ell, m} s^\ell t^m \right ) q^k.
\end{eqnarray}
Using the arguments similar to those in Sect.~6 of \cite{LQ3},
we conclude that
\begin{eqnarray}   \label{e_Q}
\sum_{m=0}^{+\infty} e \big (
  \Quot^m_{\mathcal O^{\oplus r}_Y}; s, t \big ) q^m 
= \text{\rm exp} \left ( \sum_{m=1}^{+\infty} {1 \over m}
e(Y; s^m, t^m) \mathfrak Q_{n, r}(q^m, s^m, t^m) \right ).
\end{eqnarray}
In particular, setting $s=t=1$ and using (\ref{hodge_i}),
we obtain (\ref{chi_Q}).
\end{proof}

Next, we use a torus action to compute $\chi \big (
\Quot^m_{\mathcal O^{\oplus r}}(\C^n, O) \big )$. Let 
\begin{eqnarray}   \label{Tnr}  
{\bf T}_0 = (\C^*)^{n+r}. 
\end{eqnarray}
Then ${\bf T}_0$ acts on $\Quot^m_{\mathcal O^{\oplus r}_{\C^n}}$ 
as follows. On one hand, the $n$-dimensional torus ${\bf T}_1 
:= (\C^*)^n$ acts on $\C^n$. This induces a ${\bf T}_1$-action 
on $\Quot^m_{\mathcal O^{\oplus r}_{\C^n}}$. On the other hand,
${\bf T}_2 := (\C^*)^r$ acts on $\mathcal O^{\oplus r}_{\C^n}$ via 
$(\C^*)^r \subset {\rm Aut}\big ( \mathcal O^{\oplus r}_{\C^n} \big )$.
This induces a ${\bf T}_2$-action 
on $\Quot^m_{\mathcal O^{\oplus r}_{\C^n}}$.
So ${\bf T}_0 = {\bf T}_1 \times {\bf T}_2$ acts on 
$\Quot^m_{\mathcal O^{\oplus r}_{\C^n}}$.
Note that ${\bf T}_0$ preserves 
$\Quot^m_{\mathcal O^{\oplus r}}(\C^n, O)$. Therefore,
we obtain a ${\bf T}_0$-action on 
$\Quot^m_{\mathcal O^{\oplus r}}(\C^n, O)$. More precisely, let
\begin{eqnarray*}  
R = \C [z_1, \ldots, z_n]
\end{eqnarray*}   
be the affine coordinate ring of $\C^n$. 
Denote the elements in ${\bf T}_1, {\bf T}_2, {\bf T}_0$ by
\begin{eqnarray}   \label{elt}  
{\bf t}_1 = (t_{11}, \ldots, t_{1n}), \,\,
{\bf t}_2 = (t_{21}, \ldots, t_{2r}), \,\,
{\bf t}_0 = ({\bf t}_1, {\bf t}_2)
\end{eqnarray}
respectively. Then, the element ${\bf t}_1 \in {\bf T}_1$ 
acts on the ring $R$ by 
\begin{eqnarray*}   
{\bf t}_1(z_1^{i_1} \cdots z_n^{i_n}) = 
(t_{11}z_1)^{i_1} \cdots (t_{1n}z_n)^{i_n},
\end{eqnarray*} 
and the element ${\bf t}_0 = ({\bf t}_1, {\bf t}_2) \in {\bf T}_0$
acts on the module $R^r$ by 
\begin{eqnarray}     \label{T_act}
{\bf t}_0(f_1, \ldots, f_r) = \big (t_{21} \cdot {\bf t}_1(f_1), 
\,\, \ldots, \,\, t_{2r} \cdot {\bf t}_1(f_r) \big ).
\end{eqnarray} 

\begin{lemma}   \label{fixed_pt_str}
Let $E \in \Quot^m_{\mathcal O^{\oplus r}}(\C^n, O)$.
Let $R^{(i)}$ be the $i$-th component of $R^r$, 
and let $I_{Z_i} = E \cap R^{(i)} \subset R^{(i)} \cong R$.
Then, $E \in \Quot^m_{\mathcal O^{\oplus r}}(\C^n, O)^{{\bf T}_0}$
if and only if
\begin{eqnarray}   \label{fixed_pt_str.1}
E = I_{Z_1} \oplus \cdots \oplus I_{Z_r},
\end{eqnarray}
$Z_i \in \Hilb^{\ell(Z_i)}(\C^n, O)^{{\bf T}_1}$ for 
all $1 \le i \le r$, and $\ell(Z_1) + \ldots + \ell(Z_r) = m$.
\end{lemma}
\begin{proof}
It is clear that if $E$ is of the form (\ref{fixed_pt_str.1}), 
then $E \in \Quot^m_{\mathcal O^{\oplus r}}(\C^n, O)^{{\bf T}_0}$.
Conversely, let $E \in \Quot^m_{\mathcal O^{\oplus r}}(\C^n, O)^{{\bf T}_0}$.
Since both $E$ and $R^{(i)}$ are ${\bf T}_0$-invariant, $I_{Z_i}$ is 
${\bf T}_0$-invariant. So $I_{Z_i}$ is ${\bf T}_1$-invariant, 
and $Z_i \in \Hilb^{\ell(Z_i)}(\C^n, O)^{{\bf T}_1}$
which consists of finitely many points.
Let ${\bf 1}$ denote the identity element in ${\bf T}_1$. 
Since $E$ is invariant by $\{ {\bf 1} \} \times {\bf T}_2 
\subset {\bf T}_0$, we see that $E$ is the span of elements of 
the form $f {\bf e}_i \in E$
where $f \in R$ and $\{ {\bf e}_1, \ldots, {\bf e}_r \}$ 
is the standard basis of $\C^r$. In particular, 
$$
E = I_{Z_1} + \cdots + I_{Z_r}.
$$
So $E = I_{Z_1} \oplus \cdots \oplus I_{Z_r}$.
Finally, $\ell(Z_1) + \ldots + \ell(Z_r) = m$
since $\dim R^r/E = m$.
\end{proof}

\begin{definition}  \label{r-part}
Let $n \ge 2$ and $m \ge 0$. {\it An $n$-dimensional 
partition} of $m$ is an array 
$(m_{i_1, \ldots, i_{n-1}})_{i_1, \ldots, i_{n-1}}$
of nonnegative integers $m_{i_1, \ldots, i_{n-1}}$ 
indexed by the tuples 
\begin{eqnarray}   \label{r-part.2}
(i_1, \ldots, i_{n-1}) \in (\mathbb Z_{\ge 0})^{n-1}
\end{eqnarray}
such that $m_{i_1, \ldots, i_{n-1}} \ge m_{j_1, \ldots, j_{n-1}}$
whenever $i_1 \le j_1, \ldots, i_{n-1} \le j_{n-1}$, and that
\begin{eqnarray}   \label{r-part.3}
\sum_{i_1, \ldots, i_{n-1}} m_{i_1, \ldots, i_{n-1}} = m.
\end{eqnarray}
\end{definition}

\begin{theorem}  \label{thm_Q}
Let $Y$ be a smooth projective variety of dimension $n$. 
Then,
\begin{eqnarray}   \label{thm_Q.1}
\sum_{m=0}^{+\infty} \chi(\Quot^m_{\mathcal O^{\oplus r}_Y}) q^m 
= \left ( \sum_{m=0}^{+\infty} P_n(m) q^m
  \right )^{r \cdot \chi(Y)}
\end{eqnarray}
where $P_n(m)$ denotes the number of $n$-dimensional partitions of $m$.
\end{theorem}
\begin{proof}
We conclude from formula (\ref{C*}) and Lemma~\ref{fixed_pt_str} that
\begin{eqnarray*} 
   \sum_{m=0}^{+\infty} \chi \big 
   (\Quot^m_{\mathcal O^{\oplus r}}(\C^n, O) \big ) q^m   
&=&\sum_{m=0}^{+\infty} \chi \big 
   (\Quot^m_{\mathcal O^{\oplus r}}(\C^n, O)^{{\bf T}_0} \big ) q^m \\
&=&\left ( \sum_{m=0}^{+\infty} \chi \big 
   (\Hilb^m(\C^n, O)^{{\bf T}_1} \big ) q^m \right )^r     \\
&=&\left ( \sum_{m=0}^{+\infty} \chi \big 
   (\Hilb^m(\C^n, O) \big ) q^m \right )^r.
\end{eqnarray*}
The Euler characteristic $\chi \big (\Hilb^m(\C^n, O) \big )$ is given 
by the formula:
\begin{eqnarray}   \label{5.1}
  \sum_{m=0}^{+\infty} \chi \big (\Hilb^m(\C^n, O) \big ) q^m
= \sum_{m=0}^{+\infty} P_n(m) q^m
\end{eqnarray}
(see Proposition~5.1 in \cite{Che}). 
Now (\ref{thm_Q.1}) follows from Lemma~\ref{Q_pQ}.
\end{proof}

The generating series for $P_3(m), m \ge 0$ is given by 
the McMahon function:
\begin{eqnarray}   \label{n=3}
\sum_{m=0}^{+\infty} P_3(m) q^m = M(q) 
:= \prod_{m=1}^{+\infty} {1 \over (1-q^m)^m}.
\end{eqnarray}

\begin{corollary}  \label{Q_3fold}
Let $Y$ be a smooth projective complex $3$-fold. Then,
\begin{eqnarray}   \label{Q_3fold.1}
\sum_{m=0}^{+\infty} \chi(\Quot^m_{\mathcal O^{\oplus r}_Y}) q^m 
= M(q)^{r \cdot \chi(Y)}.
\end{eqnarray}
\end{corollary}
\begin{proof}
Follows immediately  from Theorem~\ref{thm_Q} and (\ref{n=3}).
\end{proof}
\section{\bf Donaldson-Thomas invariants of certain Calabi-Yau $3$-folds}
\label{sect_Donaldson}

In this section, we compute the Donaldson-Thomas invariants for two types
of Calabi-Yau $3$-folds. The first type comes from \cite{DT, Tho} and is
studied in Subsect.~\ref{subsect_DonaldsonI}. The second type comes from 
\cite{LQ2} and is studied in Subsect.~\ref{subsect_DonaldsonII}.

\subsection{\bf Donaldson-Thomas invariants and weighted Euler 
characteristics}
\label{subsect_Donaldson}
$\,$

Let $L$ be an ample line bundle on a smooth
projective variety $Y$ of dimension $n$, and $V$ be a rank-$r$
torsion-free sheaf on $Y$. We say that $V$ is {\it (slope)
$L$-stable} if
\begin{eqnarray*}
\frac{c_1(F)\cdot c_1(L)^{n-1}}{\text{rank}(F)} \, < \,
\frac{c_1(V) \cdot c_1(L)^{n-1}}{r}
\end{eqnarray*}
for any proper subsheaf $F$ of $V$, and $V$ is {\it Gieseker
$L$-stable} if
\begin{eqnarray*}
\frac{\chi(F\otimes L^{\otimes k} )}{\text{rank}(F)} \, < \,
\frac{\chi(V\otimes L^{\otimes k})}{r}, \qquad k\gg 0
\end{eqnarray*}
for any proper subsheaf $F$ of $V$. Similarly, we define {\it
$L$-semistability} and {\it Gieseker $L$-semistability} by
replacing the above strict inequalities $<$ by inequalities $\le$.
For a class $c$ in the Chow group $A^*(Y)$, let $\mathfrak M_L(c)$
be the moduli space of $L$-stable rank-$2$ bundles with total
Chern class $c$, and let $\overline{\mathfrak M}_L(c)$ be the
moduli space of Gieseker $L$-semistable rank-$2$ torsion-free
sheaves with total Chern class $c$.

Next, let $(Y, L)$ be a polarized smooth Calabi-Yau $3$-fold.
Assume that all the rank-$2$ torsion-free sheaves in 
$\overline{\mathfrak M}_L(c)$ are actually Gieseker $L$-stable. 
By the Definition~3.54 and Corollary~3.39 in \cite{Tho}, 
{\it the Donaldson-Thomas invariant} 
\begin{eqnarray}  \label{dt}
\lambda(L, c) \in \mathbb Z
\end{eqnarray}
(also known as {\it the homolomorphic Casson invariant}) can 
be defined via the moduli space $\overline{\mathfrak M}_L(c)$.
By the Proposition~1.26 in \cite{BF}, $\overline{\mathfrak M}_L(c)$
admits a symmetric obstruction theory. It follows from 
the Theorem~4.18 in \cite{Beh} that
\begin{eqnarray}  \label{dt-Beh}
\lambda(L, c) = \W \chi \big ( \overline{\mathfrak M}_L(c) \big ).
\end{eqnarray}
Note that if $\overline{\mathfrak M}_L(c)$ is further assumed to 
be smooth, then
\begin{eqnarray}  \label{dt_smooth}
\lambda(L, c) = (-1)^{\dim \, \overline{\mathfrak M}_L(c)} \cdot 
\chi \big ( \overline{\mathfrak M}_L(c) \big ).
\end{eqnarray}

\subsection{\bf Donaldson-Thomas invariants, I}
\label{subsect_DonaldsonI}
$\,$

Let $Q_0$ be a smooth quadric in $\Pee^5$. Identifying $Q_0$ with 
the Grassmaniann $G(2, 4)$, we obtain universal rank-$2$ bundles 
$B_1$ and $B_2$ sitting in the exact sequence:
\begin{eqnarray}  \label{def_ab}
0 \to (B_1)^* \to (\mathcal O_{Q_0})^{\oplus 4} \to B_2 \to 0.
\end{eqnarray}
The Chern classes of $B_1$ and $B_2$ are the same, and 
\begin{eqnarray}  \label{c12_ab}
c_1(B_i) = H|_{Q_0}, \qquad c_2(B_i) = P
\end{eqnarray}
where $H$ is a hyperplane in $\Pee^5$, 
and $P$ is a plane contained in $Q_0$.

Next, let $Y$ be a smooth quartic hypersurface in $Q_0$. Then $Y$ is 
a smooth Calabi-Yau $3$-fold with $H_1(Y, \Z) = 0$ and 
$\Pic(Y) \cong \Pic(\mathbb P^n)$. Let
\begin{eqnarray*}
E_{0, i} = B_i|_Y
\end{eqnarray*}
 for $i = 1, 2$, and let $L = \mathcal O_{\Pee^5}(1)|_Y$. Fix a point 
$y_0 \in Y$. For $m \in \Z$, define
\begin{eqnarray}  \label{cmTho}
{\bf c}_m = -m[y_0] + \big ( 1 + H|_Y + P|_Y \big ) \in A^*(Y).
\end{eqnarray}
By Theorem~3.55 in \cite{Tho}, the moduli space 
$\overline{\mathfrak M}_{L}({\bf c}_0)$ is smooth and
\begin{eqnarray}  \label{BiY}
\overline{\mathfrak M}_{L}({\bf c}_0)
= \big \{ E_{0, 1}, \,\, E_{0, 2} \big \}.
\end{eqnarray}
Moreover, both $E_{0, 1}$ and $E_{0, 2}$ are $L$-stable.
It follows that $\lambda(L, {\bf c}_0) = 2$.

\begin{lemma}    \label{lma_Tho}
Let $m \ge 0$. Then $\overline{\mathfrak M}_{L}({\bf c}_{2m})$ 
is isomorphic to $\displaystyle{\Quot^m_{E_{0, 1}} \coprod 
\Quot^m_{E_{0, 2}}}$.
\end{lemma}
\begin{proof}
Let $E \in \Quot^m_{E_{0, i}}$ with $i = 1$ or $2$. 
Then we have an exact sequence:
\begin{eqnarray}     \label{m0Q_Tho}
0 \rightarrow E \rightarrow E_{0, i} \rightarrow
Q \rightarrow 0
\end{eqnarray}
where $Q$ is supported at finitely many points and $h^0(Y, Q) = m$.
Note that
\begin{eqnarray*}
c(E) = c(E_{0, i})/c(Q) = {\bf c}_0/(1+2m[y_0]) 
= -2m[y_0] + {\bf c}_0 = {\bf c}_{2m}.
\end{eqnarray*}
Also, $E$ is $L$-stable since $E_{0, i}$ is $L$-stable. Hence 
\begin{eqnarray}     \label{EmIn_Tho}
E \in \overline{\mathfrak M}_{L}({\bf c}_{2m}).
\end{eqnarray}

Conversely, let $E \in \overline{\mathfrak M}_{L}({\bf c}_{2m})$.
The same argument in the proof of Theorem~3.55 in \cite{Tho},
which uses only the first and second Chern classes of $E$,
shows that $E^{**} \cong E_{0, i}$ where $i = 1$ or $2$. 
Calculating the Chern classes from the canonical exact sequence
$0 \rightarrow E \rightarrow E^{**} \rightarrow Q \rightarrow 0$,
we get $c(Q) = 1+2m[y_0]$. So $Q$ is supported at finitely many points
with $h^0(Y, Q) = m$, and
\begin{eqnarray}           \label{EQm}
E \in \Quot^m_{E_{0, i}}.
\end{eqnarray}

It is well-known that the Grothendieck Quot-schemes are 
fine moduli spaces. So over $\Quot^m_{E_{0, i}} \times Y$, 
there exists universal exact sequence 
\begin{eqnarray*}        
0 \to \mathcal E_{m, i} \to \rho_2^*E_{0, i} \to 
\mathcal Q_{m, i} \to 0
\end{eqnarray*}
where $\rho_2$ is the second projection of 
$\Quot^m_{E_{0, i}} \times Y$. By (\ref{EmIn_Tho}), the sheaf
\begin{eqnarray}   \label{Em12}   
\mathcal E_{m, 1} \coprod \mathcal E_{m, 2}   
\end{eqnarray}
over $\Quot^m_{E_{0, 1}} \coprod \Quot^m_{E_{0, 2}}$
parametrizes a flat family of Gieseker $L$-semistable
rank-$2$ sheaves with Chern class ${\bf c}_{2m}$.
To show that (\ref{Em12}) is universal,
let $\mathcal E$ be a flat family of Gieseker $L$-semistable
rank-$2$ sheaves with Chern class ${\bf c}_{2m}$ 
parametrized by $T$. By (\ref{EQm}) and the universal property
of Quot-schemes, there is a morphism
\begin{eqnarray*}        
\psi: T \to \Quot^m_{E_{0, 1}} \coprod \Quot^m_{E_{0, 2}}
\end{eqnarray*}
such that $\mathcal E = (\psi \times {\rm Id}_Y)^*
(\mathcal E_{m, 1} \coprod \mathcal E_{m, 2})$.
Therefore, (\ref{Em12}) is a universal family.
\end{proof}

\begin{remark}    \label{rmk_lma_Tho}
From the proof of Lemma~\ref{lma_Tho}, we see that if
the moduli space $\overline{\mathfrak M}_{L}({\bf c}_{m})$
is not empty, then $m$ must be even and nonnegative.
\end{remark}

\begin{proposition}    \label{prop_e}
Let $Y$ be a smooth quartic hypersurface in the quadric $Q_0$,
and let ${\bf c}_m = -m \, [y_0] + \big ( 1 + H|_Y + P|_Y \big ) 
\in A^*(Y)$. Then, 
\begin{eqnarray}    \label{prop_e.0} 
  \sum_{m \in \Z} \chi \big (
    \overline{\mathfrak M}_{L}({\bf c}_m) \big ) \, q^m
= 2 \cdot M(q^2)^{2 \, \chi(Y)}.         
\end{eqnarray}
\end{proposition}
\begin{proof}
By Remark~\ref{rmk_lma_Tho}, $\overline{\mathfrak M}_{L}({\bf c}_m)
= \emptyset$ if $m < 0$ or $m$ is odd. So
\begin{eqnarray}   \label{prop_e.1} 
   \sum_{m \in \Z} \chi \big (
     \overline{\mathfrak M}_{L}({\bf c}_m) \big ) \, q^m
=    \sum_{m=0}^{+\infty} \chi \big (
     \overline{\mathfrak M}_{L}({\bf c}_{2m}) \big ) \, q^{2m}.
\end{eqnarray}
By Lemma~\ref{lma_Tho} and (\ref{e_lf}), 
$e \big ( \overline{\mathfrak M}_{L}({\bf c}_{2m}); s, t \big )
= 2 \cdot e \big ( \Quot^m_{\mathcal O^{\oplus 2}_Y}; s,t \big )$. 
Setting $s=t=1$ and using (\ref{hodge_i}), we conclude that
$\chi \big (\overline{\mathfrak M}_{L}({\bf c}_{2m}) \big )
= 2 \cdot \chi \big (\Quot^m_{\mathcal O^{\oplus 2}_Y} \big )$. 
Now our formula (\ref{prop_e.0}) follows immediately from 
(\ref{prop_e.1}) and Corollary~\ref{Q_3fold}.
\end{proof}

Let ${\W F}_m \subset \Quot^m_{E_{0, 1}}$ be the punctual Quot-scheme 
defined by:
\begin{eqnarray}    \label{def-punctual} 
{\W F}_m = \{ E \in \Quot^m_{E_{0, 1}}| \, E_{0, 1}/E 
\text{\, is supported at } y_0 \}.         
\end{eqnarray}
Fix a Zariski open neighborhood $Y_0$ of the point $y_0 \in Y$ such that 
$E_{0, 1}|_{Y_0} \cong \mathcal O_{Y_0}^{\oplus 2}$, and define the 
open subset $\Quot^m_{E_{0, 1}}(Y_0) \subset \Quot^m_{E_{0, 1}}$ by
\begin{eqnarray*}    
\Quot^m_{E_{0, 1}}(Y_0) = \{ E \in \Quot^m_{E_{0, 1}}| \, 
E_{0, 1}/E \text{\, is supported in } Y_0 \}.
\end{eqnarray*}
Then, ${\W F}_m \subset \Quot^m_{E_{0, 1}}(Y_0) \cong 
\Quot^m_{\mathcal O_{Y_0}^{\oplus 2}}$. Consider the embedding 
$\mathbb T = \C^* \hookrightarrow \mathbb T_2 
= (\C^*)^2 \subset \text{\rm Aut}(\mathcal O_{Y_0}^{\oplus 2})$ 
via $t \mapsto (1, t)$ and the induced $\mathbb T$-action on 
$\Quot^m_{\mathcal O_{Y_0}^{\oplus 2}}$. Arguments similar to those in 
the proof of Lemma~\ref{fixed_pt_str} show that 
$E \in \big ( \Quot^m_{\mathcal O_{Y_0}^{\oplus 2}} \big )^{\mathbb T}$
if and only if $E = I_{Z_1} \oplus I_{Z_2} \subset \mathcal O_{Y_0}
\oplus \mathcal O_{Y_0}$ where $Z_1$ and $Z_2$ are $0$-dimensional 
closed subschemes of $Y_0$ with $\ell(Z_1) + \ell(Z_2) = m$. Thus, we obtain
\begin{eqnarray}    \label{td-T-fixed} 
\big ( \Quot^m_{\mathcal O_{Y_0}^{\oplus 2}} \big )^{\mathbb T} \,\, \cong
\,\, \coprod_{i = 0}^m \Hilb^i(Y_0) \times \Hilb^{m-i}(Y_0).
\end{eqnarray}

\begin{lemma}  \label{difference}
Let $Z_1$ and $Z_2$ be $0$-dimensional closed subschemes of $Y$. Then,
\begin{eqnarray}  \label{difference.0}
  \dim \Hom (I_{Z_1}, \mathcal O_{Z_2}) + 
  \dim \Hom (I_{Z_2}, \mathcal O_{Z_1})
\equiv \ell(Z_1) + \ell(Z_2) \pmod 2.
\end{eqnarray}
\end{lemma}
\begin{proof}
We have $H^2(Y, \mathcal O_Y) \cong H^1(Y, \mathcal O_Y)^* = 0$. 
Taking cohomology from the exact sequence 
$0 \to I_{Z_2} \to \mathcal O_Y \to \mathcal O_{Z_2} 
\to 0$, we get $H^2(Y, I_{Z_2}) \cong H^1(Y, \mathcal O_{Z_2})=0$.
Thus, 
\begin{eqnarray*}
\Ext^1 (I_{Z_2}, \mathcal O_Y) \cong H^2(Y, I_{Z_2})^* = 0.
\end{eqnarray*}

Applying $\Hom(I_{Z_2}, \cdot)$ to the exact sequence
$0 \to I_{Z_1} \to \mathcal O_Y \to \mathcal O_{Z_1} \to 0$,
we obtain the following exact sequence:
\begin{eqnarray*}  
&0 \to \Hom (I_{Z_2}, I_{Z_1}) \to \Hom (I_{Z_2}, \mathcal O_Y)
  \to \Hom (I_{Z_2}, \mathcal O_{Z_1})& \\
&\to \Ext^1(I_{Z_2}, I_{Z_1}) \to \Ext^1 (I_{Z_2}, \mathcal O_Y).&
\end{eqnarray*}
Since $\Ext^1 (I_{Z_2}, \mathcal O_Y) = 0$ and 
$\Hom (I_{Z_2}, \mathcal O_Y) \cong \mathbb C$, we conclude that
\begin{eqnarray*}  
\dim \Hom (I_{Z_2}, \mathcal O_{Z_1}) = 
-\dim \Hom (I_{Z_2}, I_{Z_1}) + \dim \Ext^1(I_{Z_2}, I_{Z_1})+1.
\end{eqnarray*}
By symmetry, we have a similar formula for $\dim \Hom (I_{Z_1}, 
\mathcal O_{Z_2})$. Therefore,
\begin{eqnarray}  \label{difference.1}
& &\dim \Hom (I_{Z_2}, \mathcal O_{Z_1}) + 
   \dim \Hom (I_{Z_1}, \mathcal O_{Z_2}) \nonumber   \\
&=&-\dim \Hom (I_{Z_2}, I_{Z_1}) + \dim \Ext^1(I_{Z_2}, I_{Z_1})+2 
   \nonumber  \\
& &- \dim \Hom (I_{Z_1}, I_{Z_2}) + \dim \Ext^1(I_{Z_1}, I_{Z_2}) 
   \nonumber  \\
&\equiv&-\dim \Hom (I_{Z_2}, I_{Z_1}) + \dim \Ext^1(I_{Z_2}, I_{Z_1})
   \nonumber   \\
& &+ \dim \Ext^3(I_{Z_2}, I_{Z_1}) - \dim \Ext^2(I_{Z_2}, I_{Z_1}) 
   \pmod 2   \nonumber   \\
&\equiv&- \chi(I_{Z_2}, I_{Z_1}) \pmod 2.
\end{eqnarray}
Since $c_3(I_{Z_i}) = -c_3(\mathcal O_{Z_i}) = -2[Z_i]$,
the Hirzebruch-Riemann-Roch formula gives 
\begin{eqnarray*}  
\chi(I_{Z_2}, I_{Z_1}) 
= \int_Y \ch(I_{Z_2})^* \cdot \ch(I_{Z_1}) \cdot \text{\rm td}(T_Y)
= -\ell(Z_1) + \ell(Z_2).
\end{eqnarray*}
Combining this with (\ref{difference.1}) yields the desired formula
(\ref{difference.0}).
\end{proof}

For simplicity, regard $\Quot^m_{E_{0, 1}}(Y_0) =
\Quot^m_{\mathcal O_{Y_0}^{\oplus 2}}$. The symmetric obstruction theory
on $\Quot^m_{E_{0, 1}}$ restricts to a symmetric obstruction theory on 
$\Quot^m_{E_{0, 1}}(Y_0) = \Quot^m_{\mathcal O_{Y_0}^{\oplus 2}}$. 
This symmetric obstruction theory on $\Quot^m_{E_{0, 1}}(Y_0) = 
\Quot^m_{\mathcal O_{Y_0}^{\oplus 2}}$ is $\mathbb T$-equivariant
since the construction of the symmetric obstruction theory is stable
under base change (see \cite{BF, Tho}) and our Gieseker moduli space is 
fine. Let ${\W \nu}_m$ be the restriction of Behrend's function 
$\nu_{\Quot^m_{E_{0, 1}}}$ to ${\W F}_m$. Then 
$\nu_{\Quot^m_{\mathcal O_{Y_0}^{\oplus 2}}}|_{{\W F}_m} = {\W \nu}_m$.

\begin{lemma}  \label{td-punct}
$\chi({\W F}_m, {\W \nu}_m) = \chi({\W F}_m)$.
\end{lemma}
\noindent
{\it Proof.}
Note that ${\W F}_m \subset \Quot^m_{\mathcal O_{Y_0}^{\oplus 2}}$ is 
$\mathbb T$-invariant. For each $n \in \Z$, the subset 
$$
\{E \in {\W F}_m - ({\W F}_m)^{\mathbb T} |{\W \nu}_m(E) = n \}
$$ 
is $\mathbb T$-invariant and does not contain any fixed point. 
By (\ref{C*}), 
$$
\chi \big ( \{E \in {\W F}_m - ({\W F}_m)^{\mathbb T} |
{\W \nu}_m(E) = n \} \big ) = 0.
$$ 
By definition, $\chi({\W F}_m, {\W \nu}_m) = \sum_{n \in \Z} n \cdot 
\chi \big ( \{E \in {\W F}_m |{\W \nu}_m(E) = n \} \big )$. Thus,
\begin{eqnarray*}  
   \chi({\W F}_m, {\W \nu}_m) 
= \sum_{n \in \Z} n \cdot \chi \big ( \{E \in ({\W F}_m)^{\mathbb T} 
    |{\W \nu}_m(E) = n \} \big ).
\end{eqnarray*}
In view of (\ref{td-T-fixed}), $E \in ({\W F}_m)^{\mathbb T}$ if and only 
if $E = I_{Z_1} \oplus I_{Z_2} \subset \mathcal O_{Y_0} \oplus 
\mathcal O_{Y_0}$ where $Z_1 \in \Hilb^i(Y_0, y_0)$ and $Z_2 \in 
\Hilb^{m-i}(Y_0, y_0)$ for some integer $i$ satisfying $0 \le i \le m$. 
In this case, we obtain from Theorem~\ref{fixedset}~(ii) that
\begin{eqnarray*}
   {\W \nu}_m(E) 
&=&\nu_{\Quot^m_{\mathcal O_Y^{\oplus 2}}}(E) 
   = \nu_{\Quot^m_{\mathcal O_{Y_0}^{\oplus 2}}}(E)    \\
&=&(-1)^a \cdot \nu_{\Hilb^i(Y_0) \times \Hilb^{m-i}(Y_0)}(Z_1, Z_2)  \\
&=&(-1)^a \cdot \nu_{\Hilb^i(Y) \times \Hilb^{m-i}(Y)}(Z_1, Z_2)
\end{eqnarray*}
where $a$ is the difference between the dimensions of 
the Zariski tangent spaces:
\begin{eqnarray*}  
   a 
&=&\dim T_E \Quot^m_{\mathcal O_{Y_0}^{\oplus 2}} - \left ( 
   \dim T_{Z_1} \Hilb^i(Y_0) + \dim T_{Z_2} \Hilb^{m-i}(Y_0) \right )  \\
&=&\dim \Hom (I_{Z_1} \oplus I_{Z_2}, \mathcal O_{Z_1} \oplus 
   \mathcal O_{Z_2}) - \sum_{k=1}^2 \dim \Hom (I_{Z_k}, \mathcal O_{Z_k}) \\
&\equiv&\ell(Z_1) + \ell(Z_2) \pmod 2    \\
&\equiv&m    \pmod 2.
\end{eqnarray*}
by Lemma~\ref{difference}. Therefore, ${\W \nu}_m(E) = (-1)^m \cdot 
\nu_{\Hilb^i(Y) \times \Hilb^{m-i}(Y)}(Z_1, Z_2)$ and
\begin{eqnarray*}  
& &\chi({\W F}_m, {\W \nu}_m)    \\
&=&\sum_{n \in \Z} n \cdot \sum_{i =0}^m \chi \big ( \{(Z_1, Z_2) \in
   \Hilb^i(Y, y_0) \times \Hilb^{m-i}(Y, y_0)|\nu(Z_1, Z_2) = (-1)^m n \} \big ) \\
&=&(-1)^m \, \sum_{i =0}^m \sum_{n \in \Z} n \cdot \chi \big ( \{(Z_1, Z_2) \in
   \Hilb^i(Y, y_0) \times \Hilb^{m-i}(Y, y_0)|\nu(Z_1, Z_2) = n \} \big )  \\
&=&(-1)^m \, \sum_{i =0}^m \chi(\Hilb^i(Y, y_0) \times \Hilb^{m-i}(Y, y_0), \nu)  \\
&=&(-1)^m \, \sum_{i =0}^m \chi(\Hilb^i(Y, y_0), \nu_{\Hilb^i(Y)}) \cdot
   \chi(\Hilb^{m-i}(Y, y_0), \nu_{\Hilb^{m-i}(Y)})
\end{eqnarray*}
where $\nu$ denotes $\nu_{\Hilb^i(Y) \times \Hilb^{m-i}(Y)}$.
By the Corollary~4.3 in \cite{BF}, 
$$
\chi(\Hilb^i(Y, y_0), \nu_{\Hilb^i(Y)})= (-1)^i \chi(\Hilb^i(Y, y_0)).
$$
Combining this with $\chi(({\W F}_m)^{\mathbb T}) = \chi({\W F}_m)$,
we conclude that
\begin{equation}
\chi({\W F}_m, {\W \nu}_m) = \displaystyle{\sum_{i =0}^m \chi(\Hilb^i(Y, y_0)) 
\cdot \chi(\Hilb^{m-i}(Y, y_0))} = \chi(({\W F}_m)^{\mathbb T}) = \chi({\W F}_m).
\tag*{$\qed$}
\end{equation}

\begin{theorem}  \label{td-theorem}
Let $Y$ be a smooth quartic hypersurface in the quadric $Q_0$,
and let ${\bf c}_m = -m \, [y_0] + \big ( 1 + H|_Y + P|_Y \big ) 
\in A^*(Y)$. Then, 
\begin{eqnarray}    \label{td-theorem.0} 
  \sum_{m \in \Z} \lambda(L, {\bf c}_m) \, q^m
= 2 \cdot M(q^2)^{2 \, \chi(Y)}.         
\end{eqnarray}
\end{theorem}
\begin{proof}
By Remark~\ref{rmk_lma_Tho}, $\overline{\mathfrak M}_{L}({\bf c}_m)
= \emptyset$ if $m < 0$ or $m$ is odd. By (\ref{dt-Beh}),
\begin{eqnarray}   \label{td-theorem.1} 
  \sum_{m \in \Z} \lambda(L, {\bf c}_m) \, q^m
= \sum_{m=0}^{+\infty} \lambda(L, {\bf c}_{2m}) \, q^{2m}
= \sum_{m=0}^{+\infty} \W \chi \big ( 
   \overline{\mathfrak M}_L({\bf c}_{2m}) \big ) \, q^{2m}.
\end{eqnarray}
Adopting the proof of the Theorem~4.11 in \cite{BF}, we conclude that
\begin{eqnarray*}    
   \W \chi \big ( \Quot^m_{E_{0, 1}} \big ) 
&=&\sum_{\alpha \vdash n} |G_\alpha| \cdot \chi \big ( Y_0^{\ell(\alpha)} 
   \big ) \cdot \prod_i \chi({\W F}_{\alpha_i}, {\W \nu}_{\alpha_i}),  \\   
   \chi \big ( \Quot^m_{E_{0, 1}} \big ) 
&=&\sum_{\alpha \vdash n} |G_\alpha| \cdot \chi \big ( Y_0^{\ell(\alpha)} 
   \big ) \cdot \prod_i \chi({\W F}_{\alpha_i}).    
\end{eqnarray*}
Here, for a partition $\alpha$ of $n$, $G_\alpha$ denotes the 
automorphism group of $\alpha$, $\ell(\alpha)$ denotes the length 
of $\alpha$, and $Y_0^{\ell(\alpha)}$ denotes the open subset of 
the product $Y^{\ell(\alpha)}$ consisting of $\ell(\alpha)$-tuples
with pairwise distinct entries. By Lemma~\ref{td-punct}, 
$$
\W \chi \big ( \Quot^m_{E_{0, 1}} \big ) 
= \chi \big ( \Quot^m_{E_{0, 1}} \big ).
$$
Similarly, $\W \chi \big ( \Quot^m_{E_{0, 2}} \big ) = \chi \big ( 
\Quot^m_{E_{0, 2}} \big )$. By (\ref{td-theorem.1}) and Lemma~\ref{lma_Tho},
\begin{eqnarray*}    
  \sum_{m \in \Z} \lambda(L, {\bf c}_m) \, q^m
= \sum_{m=0}^{+ \infty} \chi \big ( 
   \overline{\mathfrak M}_L({\bf c}_{2m}) \big ) \, q^{2m}
= \sum_{m \in \Z} \chi \big ( 
   \overline{\mathfrak M}_L({\bf c}_m) \big ) \, q^m.
\end{eqnarray*}
Finally, we obtain $\displaystyle{\sum_{m \in \Z} \lambda(L, {\bf c}_m) 
\, q^m = 2 \cdot M(q^2)^{2 \, \chi(Y)}}$ from Proposition~\ref{prop_e}.
\end{proof}
\subsection{\bf Donaldson-Thomas invariants, II}
\label{subsect_DonaldsonII}
$\,$

Let $n \ge 2$ and $X= \mathbb P^1\times\mathbb P^n$.
Let $p$ be a point in $\Pee^1$, and $H$ be a hyperplane in $\Pee^n$. 
For simplicity, denote the divisor $a(\{p\} \times \Pee^n) 
+ b(\Pee^1 \times H)$ by $(a, b)$. When $a$ and $b$ are 
rational numbers, $(a, b)$ is a $\mathbb Q$-divisor and
$\mathcal O_X(a, b)$ is a $\mathbb Q$-line bundle. The divisor
$(1, r)$ is ample if and only if $r > 0$. Put
\begin{eqnarray}         \label{Lr}
L_r = \mathcal O_X(1, r).
\end{eqnarray}

Let $Y$ be a generic divisor of type $(2,2,n+1)$ in the product
\begin{eqnarray*}
Z= \mathbb P^1\times \mathbb P^1\times \mathbb P^n. 
\end{eqnarray*}
Then $Y$ is a smooth Calabi-Yau $(n+1)$-fold. 
By the Lefschetz hyperplane theorem, 
\begin{eqnarray*}
\Pic(Y) \cong \Pic( \mathbb P^1\times \mathbb P^1 \times
\mathbb P^n).
\end{eqnarray*}
Let $\pi_i$ be the projection from $\mathbb
P^1\times \mathbb P^1 \times \mathbb P^n$ to the $i$-th factor, 
and let
\begin{eqnarray} \label{pi}
\pi=(\pi_2\times \pi_3)|_Y: \,\, 
Y\to X= \mathbb P^1\times\mathbb P^n.
\end{eqnarray}
Put $\mathcal O_Y(a, b, c) = \pi_1^*\mathcal O_{\mathbb P^1}(a) 
\otimes \pi_2^*\mathcal O_{\mathbb P^1}(b) \otimes 
\pi_3^*\mathcal O_{\mathbb P^n}(c)|_Y$, and
\begin{eqnarray}         \label{LrY}
L_r^Y = \mathcal O_Y(0, 1, r) = \pi^*L_r.
\end{eqnarray}
Then the projection $\pi: Y\to X= \mathbb P^1\times\mathbb P^n$
is a ramified double covering with the ramification locus 
$B\subset X$ being a smooth divisor of type $(4, 2n+2)$. 
In particular, 
\begin{eqnarray} \label{push}
\pi_*\mathcal O_Y \cong \mathcal O_X \oplus \mathcal O_X(-2,
-n-1).
\end{eqnarray}
By the projection formula, if $b < (n+1)$, then we obtain
\begin{eqnarray}  \label{H1Y}
       H^1(Y, \pi^*\mathcal O_X(a,b)) 
&\cong&H^1(X, \mathcal O_X(a,b) \otimes \pi_*\mathcal O_Y)  
       \nonumber   \\
&\cong&H^1(X, \mathcal O_X(a,b) \oplus \mathcal O_X(a-2, b-n-1))
       \nonumber   \\
&\cong&H^1(X, \mathcal O_X(a,b)).
\end{eqnarray}

Fix $\ep_1, \ep_2 = 0, 1$, and fix a point $y_0 \in Y$. 
For $m \in \Z$, define
\begin{eqnarray}  \label{cm}
   {\bf c}_m 
&=&-m[y_0] + \big ( 1+\pi^*(-1, 1) \big ) \cdot 
   \big ( 1+\pi^*(\ep_1+1, \ep_2-1) \big )  \nonumber   \\
&=&\big ( 1-m[y_0] \big ) \cdot \big ( 1+\pi^*(-1, 1) \big ) 
   \cdot \big ( 1+\pi^*(\ep_1+1, \ep_2-1) \big ) \in A^*(Y).
   \qquad
\end{eqnarray}
Our first goal is to study the Gieseker moduli space 
$\overline{\mathfrak M}_{L_r^Y}({\bf c}_m)$.
When $m = 0$, the structure of the moduli space 
$\overline{\mathfrak M}_{L_r^Y}({\bf c}_0)$ has been 
determined in \cite{LQ2} (for convenience, we adopt the convention
that $e/0 = +\infty$ when $e > 0$):

\begin{lemma} \label{thm4.6}
{\rm (Theorem 4.6 in \cite{LQ2})}
Let $k=\displaystyle{(1+\ep_1){n+2-\ep_2 \choose n}-1}$. 
\begin{enumerate}
\item[{\rm (i)}] When $0<r < n(2-\ep_2)/(2+\ep_1)$, the moduli
space $\overline{\mathfrak M}_{L_r^Y}({\bf c}_0)$ is empty;

\item[{\rm (ii)}] When $n(2-\ep_2)/(2+\ep_1) < r <
n(2-\ep_2)/\ep_1$, $\overline{\mathfrak M}_{L_r^Y}({\bf c}_0)$ is
isomorphic to $\Pee^k$ and consists of all the bundles $E_0$ 
sitting in nonsplitting extensions:
\begin{eqnarray}   \label{thm4.6.1}
0 \rightarrow \mathcal O_Y(0, -1, 1) \rightarrow E_0 \rightarrow
\mathcal O_Y(0, \ep_1+1, \ep_2-1) \rightarrow 0.
\end{eqnarray}
Moreover, all these rank-$2$ bundles $E_0$ are $L_r^Y$-stable.
\end{enumerate}
\end{lemma}

Let $n(2-\ep_2)/(2+\ep_1) < r < n(2-\ep_2)/\ep_1$, 
and $E_0 \in \overline{\mathfrak M}_{L_r^Y}({\bf c}_0)$.
Consider the Quot-scheme $\Quot^m_{E_0}$.
Let $E \in \Quot^m_{E_0}$. Then we have an exact sequence:
\begin{eqnarray}     \label{m0Q}
0 \rightarrow E \rightarrow E_0 \rightarrow
Q \rightarrow 0
\end{eqnarray}
where $Q$ is supported at finitely many points and $h^0(Y, Q) = m$.
Note that
\begin{eqnarray*}
c(E) = c(E_0)/c(Q) = {\bf c}_0/(1+2m[y_0]) 
= -2m[y_0] + {\bf c}_0 = {\bf c}_{2m}.
\end{eqnarray*}
Also, $E$ is $L_r^Y$-stable since $E_0$ is $L_r^Y$-stable. Hence 
\begin{eqnarray}     \label{EmIn}
E \in \overline{\mathfrak M}_{L_r^Y}({\bf c}_{2m}).
\end{eqnarray}
In the following, we show that the converse also holds, i.e.,
every element in $\overline{\mathfrak M}_{L_r^Y}({\bf c}_{2m})$
is contained in $\Quot^m_{E_0}$ for some 
$E_0 \in \overline{\mathfrak M}_{L_r^Y}({\bf c}_0)$.

\begin{lemma}    \label{lem_ep}
Let $n \ge 2$, $\ep_1, \ep_2 = 0, 1$ and $r < n(2-\ep_2)/\ep_1$.
Let $E \in \overline{\mathfrak M}_{L_r^Y}({\bf c}_{m})$.
\begin{enumerate}
\item[{\rm (i)}] 
Then, $r \ge n(2-\ep_2)/(2+\ep_1)$ and $E$ sits in an extension
\begin{eqnarray}\label{lem_ep.1}
0 \rightarrow \mathcal O_Y(0, -1, 1) \otimes I_{Z_1} \rightarrow E 
\rightarrow \mathcal O_Y(0, \ep_1+1, \ep_2-1) \otimes I_{Z_2} 
\rightarrow 0
\end{eqnarray}
for some $0$-dimensional closed subschemes $Z_1$ and $Z_2$ of $Y$
satisfying
\begin{eqnarray*}
m = 2 \big (\ell(Z_1) + \ell(Z_2) \big );
\end{eqnarray*}

\item[{\rm (ii)}]
Moreover, if $r > n(2-\ep_2)/(2+\ep_1)$, then the above extension
does not split.
\end{enumerate}
\end{lemma}
\begin{proof}
Our proof is slightly modified from the proof of Lemma~4.2 
in \cite{LQ2} which handles the case $m = 0$. 
Since $c_1(E) = \pi^*(\ep_1, \ep_2)$ and  
\begin{eqnarray*}
c_2(E)= \pi^* \big ((2+\ep_1-\ep_2)[p \times H]-
(1-\ep_2)[\Pee^1 \times H^2] \big ),
\end{eqnarray*}
$(4c_2(E) - c_1(E)^2) \cdot c_1(L_{r_0}^Y)^{n-1} = 
2(2-\ep_2)r_0^{n-2}[2(2+\ep_1)r_0 -(2-\ep_2)(n-1)]$.
By the Bogomolov Inequality, $E$ is $L_{r_0}^Y$-unstable if $0 < r_0 <
(2-\ep_2)(n-1)/(2(2+\ep_1))$. Fix such an $r_0$ with $r_0 < r$.
Then there exists an exact sequence
\begin{eqnarray}\label{lem_ep.3}
0 \rightarrow \mathcal O_Y(a, b, c)\otimes I_{Z_1} \rightarrow
E\rightarrow \mathcal O_Y(-a, \ep_1-b, \ep_2-c)\otimes I_{Z_2}
\rightarrow 0
\end{eqnarray}
such that $\mathcal O_Y(a, b, c)\otimes I_{Z_1}$ destablizes $E$
with respect to $L_{r_0}^Y$, where $Z_1$ and $Z_2$ are codimension
at least two subschemes of $Y$. Therefore, 
\begin{eqnarray*}
c_1(\mathcal O_Y(a, b, c)) \cdot c_1(L_{r_0}^Y)^n
> c_1(E) \cdot c_1(L_{r_0}^Y)^n/2.
\end{eqnarray*}
A straightforward computation shows that this 
can be simplified into
\begin{eqnarray}\label{lem_ep.4}
n[(2c-\ep_2) +(n+1)a] + (2a +2b-\ep_1)r_0 > 0.
\end{eqnarray}
On the other hand, since $E$ is $L_r^Y$-semistable, we must have
\begin{eqnarray}\label{lem_ep.5}
n[(2c-\ep_2) +(n+1)a] + (2a +2b-\ep_1)r \le 0.
\end{eqnarray}
Calculating the second Chern class from the exact sequence
(\ref{lem_ep.3}), we get
\begin{eqnarray} \label{lem_ep.6}
\mathcal O_Y(a, b, c) \cdot \mathcal O_Y(-a, \ep_1-b, \ep_2-c) \le
c_2(E)
\end{eqnarray}
since $c_2(I_{Z_1})$ and $c_2(I_{Z_2})$ are effective cycles.
Regarding (\ref{lem_ep.6}) as an inequality of cycles in $Z$ and
comparing the coefficients of $[p \times p\times H]$ and
$[p\times\mathbb P^1\times H^2]$ yield
\begin{eqnarray}
&&[2a+(2b-\ep_1)](2c-\ep_2)+(n+1)a(2b-\ep_1)
  \ge -(\ep_1+2)(2-\ep_2),  \qquad\quad  \label{lem_ep.8} \\
&&[(2c-\ep_2)+2(n+1)a](2c-\ep_2)
  \ge  (2-\ep_2)^2.  \label{lem_ep.10}
\end{eqnarray}

Since $0 < r_0 < r$, we see from (\ref{lem_ep.4}) and
(\ref{lem_ep.5}) that $(2c-\ep_2)+(n+1)a > 0$ and 
\begin{eqnarray}   \label{2a2b}
(2a+2b-\ep_1) < 0.
\end{eqnarray} 
By (\ref{lem_ep.10}), $(2c-\ep_2)+2(n+1)a$ and $(2c-\ep_2)$
have the same sign, and so must be both positive. In particular,
$c \ge 1$. By (\ref{lem_ep.8}),
\begin{eqnarray} \label{lem_ep.10.1}
(n+1)a(2b-\ep_1) &\ge& -[2a+(2b-\ep_1)](2c-\ep_2)-
(\ep_1+2)(2-\ep_2).
\end{eqnarray}
In the following, we consider the cases $\ep_1 = 0$ and 
$\ep_1 = 1$ separately.

Assume $\ep_1 = 0$. Using (\ref{lem_ep.10.1}) and (\ref{2a2b}), 
we obtain $(n+1)a(2b) \ge 0$.
Together with (\ref{2a2b}) one more time, this implies 
either $a<0$ and $b \le 0$, or $a=0$ and $b<0$. 
If $a<0$ and $b \le 0$, then we see from (\ref{lem_ep.8}) that 
\begin{eqnarray*}
     -(2-\ep_2) 
&\le&(a+b)(2c-\ep_2)+(n+1)ab   \\
&=  &a(2c-\ep_2) +[(2c-\ep_2)+(n+1)a]b \\
&\le&a(2c-\ep_2) \le -(2c-\ep_2) \\
&\le&-(2-\ep_2). 
\end{eqnarray*}
So $a = -1$ and $c=1$, contradicting to
$(2c-\ep_2)+(n+1)a \ge 1$ and $n \ge 2$. If $a=0$ and $b<0$, then
$b(2c-\ep_2) \ge -(2-\ep_2)$ by (\ref{lem_ep.8}). 
Since $b(2c-\ep_2) \le - (2c-\ep_2) \le -(2-\ep_2)$, 
we must have $b=-1$ and $c=1$. By (\ref{lem_ep.3}), we obtain
\begin{eqnarray*}
   c(E)
&=&c(\mathcal O_Y(0,-1,1) \otimes I_{Z_1}) \cdot 
   c(\mathcal O_Y(0,1,\ep_2-1) \otimes I_{Z_2}) \\
&=&\displaystyle{
   {c(\mathcal O_Y(0,-1,1)) \over c(\mathcal O_{Z_1}(0,-1,1))} \cdot 
   {c(\mathcal O_Y(0,1,\ep_2-1)) \over c(\mathcal O_{Z_2}(0,1,\ep_2-1))}
   }.
\end{eqnarray*}
Since $c(E) = {\bf c}_{m} = (1-m[y_0]) \cdot c(\mathcal O_Y(0,-1,1))
\cdot c(\mathcal O_Y(0,1,\ep_2-1))$, we get
\begin{eqnarray}   \label{z12y0}
c(\mathcal O_{Z_1}(0,-1,1)) \cdot c(\mathcal O_{Z_2}(0,1,\ep_2-1))
= {1 \over 1-m[y_0]} = 1 + m[y_0].
\end{eqnarray}
Thus $Z_1$ and $Z_2$ are $0$-dimensional. Hence (\ref{lem_ep.3}) 
becomes (\ref{lem_ep.1}), and $m = 2 \big (\ell(Z_1) + \ell(Z_2) \big )$.
Note from (\ref{lem_ep.5}) that 
$r \ge n(2-\ep_2)/2$. Moreover, if $r > n(2-\ep_2)/2$,
then (\ref{lem_ep.1}) does not split since 
$\mathcal O_Y(0,1,\ep_2-1) \otimes I_{Z_2}$ would destabilize $E$
with respect to $L_r^Y$, contradicting to the assumption
$E \in \overline{\mathfrak M}_{L_r^Y}({\bf c}_{m})$.

Next, assume $\ep_1 = 1$. We see from (\ref{lem_ep.10.1}) 
and (\ref{2a2b}) that
$$
(n+1)a(2b-1) \ge -(2a+2b-1)(2c-\ep_2)-3(2-\ep_2) \ge 1-6 = -5. 
$$
So $a(2b-1) \ge -1$ since $n \ge 2$. If $a(2b-1) = -1$, 
then we see from $2a+(2b- 1) < 0$ that $a = -1$ and $b = 1$. 
By (\ref{lem_ep.10.1}) again, we obtain $(2c-\ep_2) \le 
3(2-\ep_2)-(n+1) \le (5-n)$ contradicting to
$(2c-\ep_2)+2(n+1)a \ge 1$ and $n \ge 2$. Therefore, we must have
$a(2b-1) \ge 0$. Since $2a+(2b- 1) < 0$, we conclude that either
$a<0$ and $(2b- 1) \le 0$, or $a=0$ and $(2b- 1)<0$.
As in the previous paragraph, we see that $a = 0$, $b = 0$ or
$-1$. If $b = 0$, then we obtain from (\ref{lem_ep.5}) that 
$r \ge n(2c-\ep_2) \ge n(2-\ep_2)$ contradicting to our assumption 
that $r < n(2-\ep_2)$. Therefore, $b=-1$.
As in the previous paragraph again, we verify that $c = 1$, 
$Z_1$ and $Z_2$ are $0$-dimensional, 
(\ref{lem_ep.3}) becomes (\ref{lem_ep.1}),
$m = 2 \big (\ell(Z_1) + \ell(Z_2) \big )$,
and $r \ge n(2-\ep_2)/3$. Moreover, if $r > n(2-\ep_2)/3$,
then (\ref{lem_ep.1}) does not split.
\end{proof}

\begin{remark}  \label{rmk_ep1}
Let $n \ge 2$, and $\ep_1, \ep_2 = 0, 1$.
By Lemma~\ref{lem_ep}, the moduli space 
$\overline{\mathfrak M}_{L_r^Y}({\bf c}_{m})$
is empty if $0 < r < n(2-\ep_2)/(2+\ep_1)$. When
\begin{eqnarray*}
n(2-\ep_2)/(2+\ep_1) \le r < n(2-\ep_2)/\ep_1,
\end{eqnarray*}
the moduli space $\overline{\mathfrak M}_{L_r^Y}({\bf c}_{m})$
is empty if $m < 0$ or $m$ is odd.
\end{remark}

\begin{remark}  \label{rmk_ep2}
As evidenced in \cite{LQ2}, the critical value
$r = n(2-\ep_2)/(2+\ep_1)$ is equivalent to saying that
$L_r^Y$ lies on a certain wall in the ample cone of $Y$.
\end{remark}

\begin{remark}  \label{rmk_ep3}
Not every sheaf $E$ sitting in a nonsplitting extension
(\ref{lem_ep.1}) is contained in the moduli space 
$\overline{\mathfrak M}_{L_r^Y}({\bf c}_{m})$.
The reason is that we might have 
\begin{eqnarray*}
E^{**} \cong \mathcal O_Y(0, -1, 1) \oplus 
\mathcal O_Y(0, \ep_1+1, \ep_2-1),
\end{eqnarray*}
and then $E$ will have a subsheaf 
$\mathcal O_Y(0, \ep_1+1, \ep_2-1) \otimes I_Z$ 
(for some $0$-dimensional closed subscheme $Z$)
destablizing $E$ with respect to $L_r^Y$.
\end{remark}

\begin{lemma}    \label{E**}
Let $n \ge 2$. Let $\ep_1, \ep_2 = 0, 1$, and $m \ge 0$. 
Assume that
\begin{eqnarray*}
n(2-\ep_2)/(2+\ep_1) < r < n(2-\ep_2)/\ep_1. 
\end{eqnarray*}
\begin{enumerate}
\item[{\rm (i)}] 
If $E \in \overline{\mathfrak M}_{L_r^Y}({\bf c}_{2m})$,
then $E^{**} \in \overline{\mathfrak M}_{L_r^Y}({\bf c}_0)$,
$E \in \Quot^m_{E^{**}}$, and $E$ is $L_r^Y$-stable;

\item[{\rm (ii)}] 
We have $E \in \overline{\mathfrak M}_{L_r^Y}({\bf c}_{2m})$ 
if and only if $E \in \Quot^m_{E_0}$ for some rank-$2$ sheaf
$E_0 \in \overline{\mathfrak M}_{L_r^Y}({\bf c}_0)$.
Moreover, in this case, $E_0 \cong E^{**}$.
\end{enumerate}
\end{lemma}
\begin{proof}
Note that (ii) follows from (i) and (\ref{EmIn}). For (i), 
we see from Lemma~\ref{lem_ep}~(i) that the double dual $E^{**}$ 
sits in an exact sequence:
\begin{eqnarray}        \label{E**.1}
0 \rightarrow \mathcal O_Y(0, -1, 1) \rightarrow E^{**}
\rightarrow \mathcal O_Y(0, \ep_1+1, \ep_2-1) \rightarrow 0.
\end{eqnarray}
In addition, we have the canonical exact sequence:
\begin{eqnarray}        \label{E**.2}
0 \rightarrow E \rightarrow E^{**} \rightarrow Q \rightarrow 0
\end{eqnarray}
where $Q$ is supported at finitely many points.
The exact sequence (\ref{E**.1}) does not split since otherwise, 
we see from (\ref{E**.2}) that $E$ would have a subsheaf 
\begin{eqnarray*}
\mathcal O_Y(0, \ep_1+1, \ep_2-1) \otimes I_Z
\end{eqnarray*}
(for some $0$-dimensional closed subscheme $Z$) destablizing $E$ 
with respect to $L_r^Y$. By Lemma~\ref{thm4.6}~(ii), 
$E^{**} \in \overline{\mathfrak M}_{L_r^Y}({\bf c}_0)$.
Calculating the Chern classes from (\ref{E**.2}), we conclude
that $h^0(Y, Q) = m$. Therefore, $E \in \Quot^m_{E^{**}}$.
By Lemma~\ref{thm4.6}~(ii) again, $E^{**}$ is $L_r^Y$-stable.
Hence the sheaf $E$ is $L_r^Y$-stable as well.
\end{proof}

By Lemma~\ref{thm4.6}~(ii) and a standard construction (see \cite{HS}),
there exists a universal vector bundle $\mathcal E_0$ over 
$\overline{\mathfrak M}_{L_r^Y}({\bf c}_0) \times Y$ which sits in
the exact sequence:
\begin{eqnarray}        \label{uni_shf}
0 \to \rho_2^*\mathcal O_Y(0, -1, 1) \to \mathcal E_0
\to \rho_2^*\mathcal O_Y(0, \ep_1+1, \ep_2-1) \otimes 
\rho_1^*\mathcal L \to 0
\end{eqnarray}
where $\rho_1$ and $\rho_2$ are the two natural projections
of $\overline{\mathfrak M}_{L_r^Y}({\bf c}_0) \times Y$,
and $\mathcal L$ stands for the line bundle 
$\mathcal O_{\Pee^k}(-1)$ on $\Pee^k \cong 
\overline{\mathfrak M}_{L_r^Y}({\bf c}_0)$. For simplicity, put
\begin{eqnarray}   \label{e0/}
\Quot^m_{\mathcal E_0/} :=
\Quot^m_{\mathcal E_0/\overline{\mathfrak M}_{L_r^Y}({\bf c}_0) 
\times Y/\overline{\mathfrak M}_{L_r^Y}({\bf c}_0)}.
\end{eqnarray}
Note that the fiber of the natural morphism 
$\Quot^m_{\mathcal E_0/} \to 
\overline{\mathfrak M}_{L_r^Y}({\bf c}_0)$ at a point
$E_0 \in \overline{\mathfrak M}_{L_r^Y}({\bf c}_0)$ is 
canonically identified with the Quot-scheme $\Quot^m_{E_0}$.

\begin{proposition}    \label{propII}
Let $n \ge 2$. Let $\ep_1, \ep_2 = 0, 1$, and $m \ge 0$. 
Assume that
\begin{eqnarray*}
n(2-\ep_2)/(2+\ep_1) < r < n(2-\ep_2)/\ep_1. 
\end{eqnarray*}
Then the moduli space 
$\overline{\mathfrak M}_{L_r^Y}({\bf c}_{2m})$
is isomorphic to the Quot-scheme $\Quot^m_{\mathcal E_0/}$.
\end{proposition}
\begin{proof}
Follows immediately from Lemma~\ref{E**}~(ii), 
the universal property of Quot-schemes, 
and an argument similar to the proof of Lemma~\ref{lma_Tho}.
\end{proof}

\begin{proposition}    \label{prop_eII}
Let $n \ge 2$. Let $\ep_1, \ep_2 = 0, 1$, $m \ge 0$, and 
\begin{eqnarray*}
k=\displaystyle{(1+\ep_1){n+2-\ep_2 \choose n}-1}. 
\end{eqnarray*}
Assume that $n(2-\ep_2)/(2+\ep_1) < r < n(2-\ep_2)/\ep_1$.
Then, 
\begin{eqnarray}   \label{prop_eII.0} 
   \sum_{m \in \Z} \chi \big (
     \overline{\mathfrak M}_{L_r^Y}({\bf c}_m) \big ) \, q^m
= (k+1) \cdot \left ( \sum_{m=0}^{+\infty} P_n(m) q^{2m}
  \right )^{2 \cdot \chi(Y)}.
\end{eqnarray}
\end{proposition}
\begin{proof}
Note from Remark~\ref{rmk_ep1} that $\overline{\mathfrak M}_{L_r^Y}({\bf c}_{m})
= \emptyset$ if $m < 0$ or $m$ is odd. So
\begin{eqnarray}   \label{prop_eII.1} 
   \sum_{m \in \Z} \chi \big (
     \overline{\mathfrak M}_{L_r^Y}({\bf c}_m) \big ) \, q^m
=    \sum_{m=0}^{+\infty} \chi \big (
     \overline{\mathfrak M}_{L_r^Y}({\bf c}_{2m}) \big ) \, q^{2m}.
\end{eqnarray}
By Proposition~\ref{propII}, $e \big (
\overline{\mathfrak M}_{L_r^Y}({\bf c}_{2m}); s, t \big )
= e \big (\Quot^m_{\mathcal E_0/}; s, t \big )$.
Since the rank-$2$ universal sheaf $\mathcal E_0$ over 
$\overline{\mathfrak M}_{L_r^Y}({\bf c}_0) \times Y$ is locally free,
we have 
\begin{eqnarray*}
e \big (\Quot^m_{\mathcal E_0/}; s, t \big )
= e \left (\Quot^m_{\mathcal O^{\oplus 2}_{
\overline{\mathfrak M}_{L_r^Y}({\bf c}_0) \times Y}/}; s, t \right )
\end{eqnarray*}
by (\ref{e_lf}). Note that there exists a canonical isomorphism
\begin{eqnarray*}
      \Quot^m_{\mathcal O^{\oplus 2}_{
        \overline{\mathfrak M}_{L_r^Y}({\bf c}_0) \times Y}/} 
\cong \Quot^m_{\mathcal O^{\oplus 2}_Y} \times 
        \overline{\mathfrak M}_{L_r^Y}({\bf c}_0)
\end{eqnarray*}
by the universal property of Quot-schemes. 
It follows from (\ref{hodge_iii}) that
\begin{eqnarray*}
   e \big (\overline{\mathfrak M}_{L_r^Y}({\bf c}_{2m}); s, t \big )
&=&e \left (\Quot^m_{\mathcal O^{\oplus 2}_Y} \times 
     \overline{\mathfrak M}_{L_r^Y}({\bf c}_0); s, t \right )    \\
&=&e \left (\Quot^m_{\mathcal O^{\oplus 2}_Y}; s, t \right )
   \cdot e \left ( \overline{\mathfrak M}_{L_r^Y}({\bf c}_0); 
   s, t \right ). 
\end{eqnarray*}
Since $\overline{\mathfrak M}_{L_r^Y}({\bf c}_0) \cong \Pee^k$
and $e(\Pee^k; s, t)= (1-(st)^{k+1})/(1-st)$, we get
\begin{eqnarray}   \label{cor_e.1} 
  e \big (\overline{\mathfrak M}_{L_r^Y}({\bf c}_{2m}); s, t \big )
= {1-(st)^{k+1} \over 1-st} \cdot 
  e \left (\Quot^m_{\mathcal O^{\oplus 2}_Y}; s, t \right ). 
\end{eqnarray}
Setting $s=t=1$, we see that (\ref{prop_eII.0}) follows from 
(\ref{prop_eII.1}), (\ref{hodge_i}) and Theorem~\ref{thm_Q}.
\end{proof}

\begin{lemma}    \label{affine}
Let $\mathcal E_0$ be the universal sheaf over 
$\overline{\mathfrak M}_{L_r^Y}({\bf c}_0) \times Y$ sitting in
the exact sequence (\ref{uni_shf}). Fix a point $y \in Y$. 
Let $U$ be an open affine neighborhood of $y$ such that 
$\mathcal O_Y(0, -1, 1)|_U \cong \mathcal O_U$ and $\mathcal O_Y(0, 
\ep_1+1, \ep_2-1)|_U \cong \mathcal O_U$. Then,
\begin{eqnarray*}
\mathcal E_0|_{\overline{\mathfrak M}_{L_r^Y}({\bf c}_0) \times U} \cong 
\rho_1^*\left ( \mathcal O_{\overline{\mathfrak M}_{L_r^Y}({\bf c}_0)} 
\oplus \mathcal L \right )
\end{eqnarray*}
where $\rho_1: \overline{\mathfrak M}_{L_r^Y}({\bf c}_0) \times U \to 
\overline{\mathfrak M}_{L_r^Y}({\bf c}_0)$ denotes the first projection. 
\end{lemma}
\begin{proof}
Restricting the exact sequence (\ref{uni_shf}) to 
$\overline{\mathfrak M}_{L_r^Y}({\bf c}_0) \times U$ yields
$$
0 \to \mathcal O \to \mathcal E_0|_{\overline{\mathfrak M}_{L_r^Y}
({\bf c}_0) \times U} \to \rho_1^*\mathcal L \to 0
$$
where $\mathcal O$ denotes $\mathcal O_{\overline{\mathfrak M}_{L_r^Y}
({\bf c}_0) \times U}$. So it suffices to prove 
${\rm Ext}^1(\rho_1^*\mathcal L, \mathcal O) = 0$, i.e.,
\begin{eqnarray}        \label{affine.1}
H^1(\overline{\mathfrak M}_{L_r^Y}({\bf c}_0) \times U, 
\rho_1^*\mathcal L^{-1}) = 0.
\end{eqnarray}
Recall that $\overline{\mathfrak M}_{L_r^Y}({\bf c}_0) \cong \Pee^k$ 
and $\mathcal L = \mathcal O_{\Pee^k}(-1)$. 
Let $\w \rho_2: \Pee^k \times U \to U$ be the second projection.
By the Leray spectral sequence, we obtain an exact sequence
$$
0 \to H^1(U, (\w \rho_2)_*\rho_1^*\mathcal O_{\Pee^k}(1)) \to 
H^1(\Pee^k \times U, \rho_1^*\mathcal O_{\Pee^k}(1))
\to H^0(U, R^1(\w \rho_2)_*\rho_1^*\mathcal O_{\Pee^k}(1)).
$$
Since $U$ is affine, $H^1(U, (\w \rho_2)_*\rho_1^*\mathcal O_{\Pee^k}(1))
= 0$. Since $R^1(\w \rho_2)_*\rho_1^*\mathcal O_{\Pee^k}(1) = 0$,
we conclude that $H^1(\Pee^k \times U, \rho_1^*\mathcal O_{\Pee^k}(1))
= 0$. This verifies (\ref{affine.1}).
\end{proof}

Now let $n = 2$. Then $Y$ is a smooth Calabi-Yau $3$-fold with 
$H_1(Y, \Z) = 0$. For the fixed point $y_0 \in Y$, let $Y_0$ be an open 
affine neighborhood of $y_0$ such that both $\mathcal O_Y(0, -1, 1)|_{Y_0}$ 
and $\mathcal O_Y(0, \ep_1+1, \ep_2-1)|_{Y_0}$ are trivial. Define 
\begin{eqnarray*} 
   {\W F}_m 
&=&\{ E \in \Quot^m_{\mathcal E_0/}| \, E^{**}/E \text{\, is supported 
         at } y_0 \},    \\
   \Quot^m_{\mathcal E_0/}(Y_0) 
&=&\{ E \in \Quot^m_{\mathcal E_0/}| \, E^{**}/E \text{\, is supported 
   in } Y_0 \}.
\end{eqnarray*}
Then ${\W F}_m \subset \Quot^m_{\mathcal E_0/}(Y_0) \cong 
\Quot^m_{\rho_1^*\big (\mathcal O_{\Pee^k} \oplus \mathcal O_{\Pee^k}(-1) 
\big )/\Pee^k \times Y_0/\Pee^k}$ by Lemma~\ref{affine}.
Consider the embedding $\mathbb T = \C^* \hookrightarrow \mathbb T_2 
= (\C^*)^2 \subset \text{\rm Aut}\big (\mathcal O_{\Pee^k} \oplus 
\mathcal O_{\Pee^k}(-1) \big )$ via $t \mapsto (1, t)$ and 
the induced $\mathbb T$-action on $\Quot^m_{\rho_1^*\big (\mathcal O_{\Pee^k} 
\oplus \mathcal O_{\Pee^k}(-1) \big )/\Pee^k \times Y_0/\Pee^k}$. 
As in Lemma~\ref{fixed_pt_str} and (\ref{td-T-fixed}), 
\begin{eqnarray}    \label{Y0II} 
\left ( \Quot^m_{\rho_1^*\big (\mathcal O_{\Pee^k} \oplus \mathcal O_{\Pee^k}(-1) 
\big )/\Pee^k \times Y_0/\Pee^k} \right )^{\mathbb T} \,\, \cong
\,\, \coprod_{i = 0}^m \Pee^k \times \Hilb^i(Y_0) \times \Hilb^{m-i}(Y_0).
\end{eqnarray}
An argument similar to the proof of Lemma~\ref{td-punct} proves that
\begin{eqnarray}        \label{td-punctII}
\chi({\W F}_m, {\W \nu}_m) = (-1)^k \cdot \chi({\W F}_m)
\end{eqnarray}
where ${\W \nu}_m$ is the restriction of Behrend's function
$\nu_{\Quot^m_{\mathcal E_0/}}$ to ${\W F}_m$.

\begin{theorem}  \label{td-theoremII}
Let $Y \subset \mathbb P^1\times \mathbb P^1\times \mathbb P^2$ be
a generic smooth Calabi-Yau hypersurface.
Let $\ep_1, \ep_2 = 0, 1$, and $k=(1+\ep_1)(4-\ep_2)(3-\ep_2)/2-1$.
Let $\pi: Y \to \mathbb P^1\times \mathbb P^2$ be the restriction 
to $Y$ of the projection of $\mathbb P^1\times \mathbb P^1 \times 
\mathbb P^2$ to the product of the last two factors.
Fix a point $y_0 \in Y$, and define in $A^*(Y)$ the class:
\begin{eqnarray*}
{\bf c}_m = -m[y_0] + \big ( 1+\pi^*(-1, 1) \big ) \cdot 
   \big ( 1+\pi^*(\ep_1+1, \ep_2-1) \big ).
\end{eqnarray*}
\begin{enumerate}
\item[{\rm (i)}] If $0 < r < 2(2-\ep_2)/(2+\ep_1)$,
then $\lambda(L_r^Y, {\bf c}_{m}) = 0$ for all $m \in \Z$.

\item[{\rm (ii)}] If $2(2-\ep_2)/(2+\ep_1) < r < 2(2-\ep_2)/\ep_1$, then
\begin{eqnarray}      \label{td-theoremII.0}
  \sum_{m \in \Z} \lambda(L_r^Y, {\bf c}_m) q^m
= (-1)^{k} \cdot (k+1) \cdot M(q^2)^{2 \, \chi(Y)}.
\end{eqnarray}
\end{enumerate}
\end{theorem}
\noindent
{\it Proof.}
(i) In this case, $\overline{\mathfrak M}_{L_r^Y}({\bf c}_{m})
= \emptyset$ by Remark~\ref{rmk_ep1}. 
Hence $\lambda(L_r^Y, {\bf c}_{m}) = 0$.

(ii) By Remark~\ref{rmk_ep1}, Proposition~\ref{propII} and (\ref{dt-Beh}),
we have
\begin{eqnarray}      \label{td-theoremII.1}
  \sum_{m \in \Z} \lambda(L_r^Y, {\bf c}_m) q^m
= \sum_{m=0}^{+ \infty} \lambda(L_r^Y, {\bf c}_{2m}) q^{2m}
= \sum_{m=0}^{+ \infty} \W \chi \big ( \Quot^m_{\mathcal E_0/} \big ) \, q^{2m}.
\end{eqnarray}
By Lemma~\ref{affine}, we can adopt the proof of the Theorem~4.11 
in \cite{BF}. So
\begin{eqnarray*}    
   \W \chi \big ( \Quot^m_{\mathcal E_0/} \big ) 
&=&\sum_{\alpha \vdash n} |G_\alpha| \cdot \chi \big ( Y_0^{\ell(\alpha)} 
   \big ) \cdot \prod_i \chi({\W F}_{\alpha_i}, {\W \nu}_{\alpha_i}),  \\   
   \chi \big ( \Quot^m_{\mathcal E_0/} \big ) 
&=&\sum_{\alpha \vdash n} |G_\alpha| \cdot \chi \big ( Y_0^{\ell(\alpha)} 
   \big ) \cdot \prod_i \chi({\W F}_{\alpha_i}).    
\end{eqnarray*}
By (\ref{td-punctII}), $\W \chi \big ( \Quot^m_{\mathcal E_0/} \big )
= (-1)^k \cdot \chi \big ( \Quot^m_{\mathcal E_0/} \big )$.
By (\ref{td-theoremII.1}) and Proposition~\ref{prop_eII},
\begin{equation}
  \sum_{m \in \Z} \lambda(L_r^Y, {\bf c}_m) q^m
= (-1)^{k} \cdot (k+1) \cdot M(q^2)^{2 \, \chi(Y)}.
\tag*{$\qed$}
\end{equation}

\end{document}